\documentclass{article}
\usepackage[utf8]{inputenc}
\usepackage{amsmath,amsfonts,amssymb,amsthm, graphicx,biblatex,braket}

\newtheorem{theorem}{Theorem}[section]
\newtheorem{corollary}[theorem]{Corollary}
\newtheorem{lemma}[theorem]{Lemma}
\theoremstyle{definition}
\newtheorem{definition}[theorem]{Definition}

\newtheorem{example}[theorem]{Example}

\DeclareMathOperator{\Ap}{Ap}
\DeclareMathOperator{\F}{F}
\DeclareMathOperator{\g}{g}
\DeclareMathOperator{\e}{e}
\DeclareMathOperator{\m}{m}

\title{Numerical Semigroups Generated by Quadratic Sequences}
\author{Mara Hashuga \and Megan Herbine \and Alathea Jensen}

\begin{document}

\maketitle

\begin{abstract}
We investigate numerical semigroups generated by any quadratic sequence with initial term zero and an infinite number of terms.  We find an efficient algorithm for calculating the Ap\'ery set, as well as bounds on the elements of the Ap\'ery set.  We also find bounds on the Frobenius number and genus, and the asymptotic behavior of the Frobenius number and genus.  Finally, we find the embedding dimension of all such numerical semigroups.
\end{abstract}

\section{Introduction}

The investigation of numerical semigroups generated by particular kinds of sequences dates back to at least 1942, when Brauer \cite{brauer42} found the Frobenius number for numerical semigroups generated by sequences of consecutive integers.  Roberts \cite{roberts56} followed in 1956 with the Frobenius number of numerical semigroups generated by generic arithmetic sequences.

It might seem natural that after conquering arithmetic sequences, work would proceed apace on other common types of sequences, especially geometric sequences and polynomial sequences, which are the other two types of sequences most frequently encountered in mathematics education.  However, this was not the case.

Instead, reseachers such as Lewin \cite{lewin75} and Selmer \cite{selmer77} turned their attention to generalized arithmetic sequences---sequences which are arithmetic except for one term.  Work on generalized arithmetic sequences continues to the present day, in, for example, \cite{omidali12}, \cite{chapman17}, and \cite{lee19}.

Work on geometric sequences did not appear in the literature until 2008, when Ong and Ponomarenko \cite{ong08} found the Frobenius number of a numerical semigroup generated by a geometric sequence.  Work on generalized geometric sequences, called compound sequences, also continues to the present day, in, for example, \cite{kiers16}.  Some work has also been done on other, more exotic types of sequences, such as the Fibonacci sequence \cite{marin07}, sequences of repunits \cite{rosales16}, sequences of Mersenne numbers \cite{rosales17}, and sequences of Thabit numbers \cite{rosales15}.

Only a very small amount of work has appeared on numerical semigroups generated by polynomial sequences, and only for particular instances of polynomials, not for generic polynomials.  This includes numerical semigroups generated by three consecutive squares or cubes \cite{lepilov15}, infinite sequences of squares \cite{moscariello15}, and sequences of three consecutive triangular numbers or four consecutive tetrahedral numbers \cite{perez18}.  In \cite{jurgen17}, the authors tantalizingly defined something called a quadratic numerical semigroup; however, the quadratic object in question is an associated algebraic ideal, not a sequence of generators.

Thus, to date, no one has investigated the numerical semigroups generated by a generic quadratic sequence, a generic cubic sequence, nor any generic polynomial sequence of higher degree.  This work is important not only because these are common sequences worthy of investigation in their own right, but also because every numerical semigroup is generated by a subset of a polynomial sequence of sufficient degree.  This is so because a polynomial formula can be fitted to any finite set of numbers.  Hence, an understanding of numerical semigroups generated by polynomial sequences would contribute to the understanding of all numerical semigroups.

In this article, we begin the investigation of numerical semigroups generated by generic quadratic sequences, and lay out a framework for its continuation.

\section{Background: Numerical Semigroups}

In this section, we will define the most important objects and parameters associated with numerical semigroups, as well as common facts about these objects, given here as lemmas.  These definitions and lemmas are taken from the standard reference text \cite{rosales09}.

Before we begin, it is very important to note that throughout this article, we will use $\mathbb{N}$ to denote $\mathbb{N}=\Set{1,2,3,\ldots}$, and $\mathbb{N}_0$ to denote $\mathbb{N}_0=\Set{0,1,2,3,\ldots}$.

A \textbf{monoid} is a set $M$, together with a binary operation $+$ on $M$, such that $+$ is closed, associative, and has an identity element in $M$.  A subset $N$ of $M$ is a \textbf{submonoid} of $M$ if and only if $N$ is also a monoid using the same operation as $M$.

Given a monoid $M$ and a subset $A$ of $M$, the smallest submonoid of $M$ containing $A$ is 
\[\langle A\rangle =\Set{\lambda_{1}a_1+\ldots+\lambda_na_n | n\in \mathbb{N}_0, \lambda_1, \ldots, \lambda_n \in \mathbb{N}_0 \text{ and } a_1, \ldots,a_n \in A}.\]
The elements of $A$ are called \textbf{generators} of $\langle A\rangle$ or a \textbf{system of generators} of $\langle A\rangle$, and we accordingly say that $\langle A\rangle$ is \textbf{generated by} $A$.

Clearly, $\mathbb{N}_0$ is a monoid under the standard addition operation.  A submonoid $S$ of $\mathbb{N}_0$ is a \textbf{numerical semigroup} if and only if it has a finite complement in $\mathbb{N}_0$.

\begin{lemma}
Let $A$ be a nonempty subset of $\mathbb{N}_0$. Then $\langle A\rangle$ is a numerical semigroup if and only if $\gcd(A)=1$.
\end{lemma}

A system of generators of a numerical semigroup is said to be \textbf{minimal} if and only if none of its proper subsets generate the numerical semigroup.

\begin{lemma}
Every numerical semigroup has a unique, finite, minimal system of generators.  Furthermore, any set which generates the numerical semigroup contains this minimal system of generators as a subset.
\end{lemma}

The least element in the minimal system of generators of a numerical semigroup $S$ is called the \textbf{multiplicity} of $S$, and is denoted by $\m(S)$.  The cardinality of the minimal system of generators is called the \textbf{embedding dimension} of $S$ and is denoted by $\e(S)$.

\begin{lemma}\label{multembedlemma}
Let $S$ be a numerical semigroup. Then $\m(S)=\min(S \setminus \Set{0})$ and $\e(S) \leq \m(S)$.
\end{lemma}

The greatest integer not in a numerical semigroup $S$ is known as the \textbf{Frobenius number} of $S$ and is denoted by $\F(S)$.
The set of elements in $\mathbb{N}_0$ that are not in $S$ is known as the \textbf{gap set} of $S$, and is denoted by ${\rm G}(S)$.
The cardinality of the gap set is known as the \textbf{genus} of $S$ and is denoted by $\g(S)$.

The \textbf{Ap\'ery Set} of $n$ in $S$, where $n$ is a nonzero element of the numerical semigroup $S$, is
\[\Ap(S,n)=\Set{s\in S\mid s-n\notin S}\]

\begin{lemma}
Let $S$ be a numerical semigroup and let $n$ be a nonzero element of $S$. Then $\Ap(S,n)=\Set{0=w(0), w(1), \ldots ,w(n-1)}$, where $w(i)$ is the least element of $S$ congruent with $i$ modulo $n$, for all $i\in \Set{0, \ldots ,n-1}$.
\end{lemma}

There is no known general formula for the Frobenius number or the genus for numerical semigroups. However, we can compute both values if the Apéry set of any nonzero element of the semigroup is known.

\begin{lemma}\label{FgLemma}
Let $S$ be a numerical semigroup and let $n$ be a nonzero element of $S$. Then \[\F(S)=(\max{\Ap(S,n)})-n\]
and
\[\g(S)= \frac{1}{n} \left(\sum_{w\in \Ap(S,n)}w\right)- \frac{n-1}{2}\]
\end{lemma}

\section{Generating a Numerical Semigroup from a \\ Quadratic Sequence}

In this section, we will establish definitions and notation for the particular kind of numerical semigroups that we are investigating.

A \textbf{quadratic sequence} is a sequence whose terms are given by a quadratic function $y_n=c_0+c_1 n+c_2 n^2$, where $n\in\mathbb{N}_0$ and $c_0,c_1,c_2\in\mathbb{R}$.  Clearly, there are several associated parameters that will affect a numerical semigroup generated by a quadratic sequence: namely, the constants $c_0,c_1,c_2$, but also the number of terms from the quadratic sequence that are used as generators.

In this paper, we will only study numerical semigroups generated by infinite quadratic sequences with initial term $y_0=0$.  However, it would be interesting in future to study numerical semigroups generated by quadratic sequences in greater generality.

Below, we define precisely the numerical semigroups that we will investigate in this paper.

\begin{definition}
We say that a numerical semigroup $S$ is \textbf{generated by an infinite quadratic sequence with initial term zero} if and only if there exist some $c_0,c_1,c_2\in\mathbb{R}$ such that $y_n=c_0+c_1 n+c_2 n^2$ for all $n\in\mathbb{N}_0$ and $y_0=0$ and $S=\langle y_0,y_1,y_2,\ldots \rangle$.  We denote the set of all numerical semigroups generated by an infinite quadratic sequence with initial term zero by the name $\mathcal{Q}_0^\infty$.
\end{definition}

Clearly, since we are choosing to set $y_0=0$, we must have $c_0=0$.  We would now like to specify conditions on $c_1$ and $c_2$ that guarantee both that $S=\langle y_0,y_1,y_2,\ldots \rangle$ is a numerical semigroup, as well as that $S$ could be any numerical semigroup in $\mathcal{Q}_0^\infty$.

First and most obvious, we need for all the $y_n$ terms to be in $\mathbb{N}_0$ in order for $S$ to be a numerical semigroup.  It might be tempting to suppose that all the terms of the sequence $y_n$ are in $\mathbb{N}_0$ if and only if $c_0, c_1, c_2$ are in $\mathbb{N}_0$.  However, this is not the case.  For example, when $c_1=-1.5$ and $c_2=2.5$, $y_n\in\mathbb{N}_0$ for all $n\in\mathbb{N}_0$.

Because of this difficulty, we have chosen to express the formula for our quadratic sequence $y_n$ in quite a different form than $y_n=c_0+c_1 n+c_2 n^2$.  We will first define what it means for a sequence to be quadratic in an alternative manner, and then build up to the formula for $y_n$ from there.

It is a well known fact that a sequence $y_n$ is quadratic if and only if its sequence of first differences $y_{n}-y_{n-1}$ is arithmetic.  Moreover, a sequence of elements in $\mathbb{N}_0$ is quadratic if and only if its sequence of first differences is in $\mathbb{N}_0$ and is arithmetic.

We will use a sequence of first differences, called $x_n$, to define our quadratic sequence $y_n$.  Let $x_n=a+bn$, where $a,b\in\mathbb{N}_0$. Then, in terms of $x_n$, our quadratic sequence, which will be denoted $y_n$, is defined by $y_n=y_{n-1}+x_{n-1}$, with initial term $y_0=0$.

\begin{lemma}
If $y_0=0$, then $y_n=na+{n\choose 2}b$.
\end{lemma}
\begin{proof}
Let $y_0=0$. We know that $y_n=y_{n-1}+x_{n-1}$ and that $x_n=a+bn$. So, $y_1=y_0+x_0=0+a=a=1a+{1 \choose 2}b$. 

Next, assume that $y_n=na+{n \choose 2}b$ for $n=k$. We will show that when this is the case, it will also be true for $n=k+1$. Since $y_{k+1}=y_k+x_k$, we know that $y_{k+1}=ka+{k \choose 2}b+a+kb=(k+1)a+\frac{k(k-1)}{2}b+kb=(k+1)a+\frac{k^2-k+2k}{2}b=(k+1)a+\frac{k(k+1)}{2}b=(k+1)a+{{k+1} \choose 2}b$. Therefore, $y_n=na+{n \choose 2}b$ for $n\geq 0$.
\end{proof}

Now that we have an expression for our quadratic sequence, we would like to know when this sequence generates a numerical semigroup.  First, we will establish some notation.

Regardless of whether it generates a numerical semigroup, we can always use the quadratic sequence $y_n$ to generate a monoid.  Clearly the elements of the monoid depend on $a$ and $b$, so we will refer to the monoid generated by $y_n$ for the particular values of $a$ and $b$ as $S(a,b)$.  The following definition gives this notation more formally.

\begin{definition}
Let $a,b\in\mathbb{N}_0$.  Then
\[S(a,b):=\Braket{na+{n \choose 2}b | n\in\mathbb{N}_0}\]
\end{definition}

At times, we may refer to $S(a,b)$ simply as $S$, if the values of $a$ and $b$ are fixed and are clear from context.

It is clear from the definition that $S(a,b)$ is always a submonoid of $\mathbb{N}_0$, however, we would like to know when $S(a,b)$ is a numerical semigroup.

\begin{theorem}
For all $a,b\in\mathbb{N}_0$, $S(a,b)$ is a numerical semigroup if and only if $\gcd(a,b)=1$.
\begin{proof}
Let $S=\langle y_0, y_1, y_2, \ldots \rangle$ such that $y_n=na+\frac{n(n-1)}{2}b$ and let $A=\Set{y_0, y_1, y_2, \ldots}$. Assume $\gcd(a,b)\neq1$. This means $\gcd(a,b)=w$ for some $w\in \mathbb{Z}$. So $a=wc$ and $b=wd$ for $c,d\in \mathbb{Z}$. Now we can say $y_n=n(wc)+\frac{n(n-1)}{2}(wd)=w\left(nc+\frac{n(n-1)}{2}d\right)$. So if $\gcd(a,b)=w$, all elements of the generating set of $S$ will have a factor of $w$, meaning that if $\gcd(a,b)\neq1$, then $\gcd(A)\neq1$. So if $\gcd(A)=1$, then $\gcd(a,b)=1$.

Now, assume $\gcd(A)\neq1$. This means all elements of the generating set of $S$ share a common factor, say $u\in \mathbb{Z}$. We know that $y_1=a$ and $y_2=2a+b$ are elements of $A$. So, $a=ku$ and $2a+b=\ell u$ for $k,\ell\in\mathbb{Z}$. We can substitute $a=ku$ into $2a+b=\ell u$ and we get $2(ku)+b=\ell u$. Simplifying, we obtain $b=u(\ell -2k)$. Since $\ell -2k\in\mathbb{Z}$, $b$ is divisible by $u$. Since we showed that $a$ and $b$ are both divisible by $u$, we can say that $\gcd(a,b)\neq1$. Therefore, if $\gcd(A)\neq1$, then $\gcd(a,b)\neq1$, meaning that if $\gcd(a,b)=1$, then $\gcd(A)=1$.
\end{proof}
\end{theorem}

Note that there are three special cases of numerical semigroups in $\mathcal{Q}_0^\infty$ that will be excluded from some the theorems in the remainder of this paper, because some of our techniques and formulas do not work on them.  These special cases are when $a=0$, when $a=1$, and when $b=0$.

When $a=0$, in order to have $\gcd(a,b)=1$, we must have $b=1$.  In this case, $y_2=2a+{2\choose 2}b=1$, hence $S(0,1)=\mathbb{N}_0$.  Similarly, when $a=1$, we have $y_1=1a+{1\choose 2}b=1$, so $S(1,b)=\mathbb{N}_0$.  Finally, when $b=0$, in order to have $\gcd(a,b)=1$, we must have $a=1$, and so $S(1,0)=\mathbb{N}_0$ as well.

Let us look at an example of a numerical semigroup in $\mathcal{Q}_0^\infty$ to clarify the definitions and concepts we have just discussed.

\begin{example}
Let $a=2$ and $b=1$.  Then
\[y_n=na+{n \choose 2}b=2n+\frac{n(n-1)}{2}=\frac{n(n+3)}{2}\]
Hence,
\[S(2,1)=\langle y_0,y_1,y_2,\dots\rangle=\langle 0, 2, 5, 9, 14, 20, 27, \dots\rangle\]
As $S(2,1)$ includes $2$, it must include all even natural numbers, and, as it includes $5$ and $2$, it must include all odd numbers beginning with $5$.  In fact, $1$ and $3$ are the only natural numbers that cannot be made with these generators, hence
\[S(2,1)=\langle 2, 5\rangle=\mathbb{N}_0\setminus \Set{1, 3}\]
\end{example}

\section{The $\mu_{a,b}$ Sequence and the Ap\'ery Set}

One of the most important questions that we can ask about $S(a,b)$ is the following: for a given coefficient $n$ of $b$, what is the minimum coefficient $m$ of $a$ such that $ma+nb\in S(a,b)$?  We will use the notation $\mu_{a,b}(n)$ to denote the answer to this question.  The following definition formalizes this notion.

\begin{definition}
Let $S(a,b)\in\mathcal{Q}_0^\infty$ and let $n\in\mathbb{Z}$.  Then
\[\mu_{a,b}(n):=\min\Set{m\in\mathbb{Z} | ma+nb\in S(a,b)}\]
\end{definition}

Note that we have defined the quantity $\mu_{a,b}$ as being a map on $\mathbb{Z}$ rather than on $\mathbb{N}_0$.  This is so because it is possible for $ma+nb$ to be in $S(a,b)$ when $m$ or $n$ are negative, although not, of course, when both $m$ and $n$ are negative.  For example, when $m=b$ and $n=-a$, $ma+nb=0\in S(a,b)$, and likewise when $m=-b$ and $n=a$.

The $\mu_{a,b}$ values are very important, because they give us the Ap\'ery set $\Ap(S(a,b),a)$, as shown by the following theorem.

\begin{theorem}\label{AperyS(a,b)}
For all $S(a,b)\in\mathcal{Q}_0^\infty$ where $a\geq 1$,
\[\Ap(S(a,b),a)=\Set{\mu_{a,b}(n)a+nb | n=0,1,2,\dots,a-1}\]
\end{theorem}
\begin{proof}
We assumed $\gcd(a,b)=1$, so $0b,1b,\ldots,(a-1)b$ form all of the different congruence classes modulo $a$. It follows that all elements of the form $\mu_{a,b}(n)a+nb$ where $n=0,1,2,\dots,a-1$ are in different congruence classes modulo $a$.

We defined $\mu_{a,b}(n)=\min\Set{m\in\mathbb{Z} | ma+nb\in S(a,b)}$, so $\mu_{a,b}(n)a+nb-a=(\mu_{a,b}(n)-1)a+nb$ is not in $S(a,b)$. Therefore, each $\mu_{a,b}(n)a+nb$ is the smallest element of $S(a,b)$ in its congruence class modulo $a$, so each $\mu_{a,b}(n)a+nb$ for $n=0,1,2,\dots,a-1$ is in $\Ap(S(a,b),a)$.
\end{proof}

In fact, only the values $\mu_{a,b}(0), \mu_{a,b}(1), \ldots, \mu_{a,b}(a-1)$ are needed to find any value of $\mu_{a,b}$, as the next theorem will show.  In a later section, we will show an efficent way of calculating these values.

\begin{theorem}\label{mu(a,b) of n+ia}
For all $S(a,b)\in\mathcal{Q}_0^\infty$ and all $i,n\in\mathbb{Z}$,
\[\mu_{a,b}(n+ia)=\mu_{a,b}(n)-ib\]
\end{theorem}
\begin{proof}
Let $S(a,b)\in\mathcal{Q}_0^\infty$ and suppose $i,n\in\mathbb{Z}$. We know from the definition of $\mu_{a,b}(n)$ that $\mu_{a,b}(n+ia)=\min\Set{m\in\mathbb{Z} | ma+(n+ia)b\in S(a,b)}$. So, we can collect the copies of $a$ to get 
\[\mu_{a,b}(n+ia)=\min\Set{m\in\mathbb{Z} | (m+ib)a+nb\in S(a,b)}\]
Now we can add $ib$ to both sides to get 
\[\mu_{a,b}(n+ia)+ib=ib+\min\Set{m\in\mathbb{Z} | (m+ib)a+nb\in S(a,b)}\]
Since $ib$ is not dependent on $a$, we can move $ib$ inside the set to obtain 
\[\mu_{a,b}(n+ia)+ib=\min\Set{(m+ib)\in\mathbb{Z} | (m+ib)a+nb\in S(a,b)}\]
Now, let $m'=m+ib$. Then,
\[\mu_{a,b}(n+ia)+ib=\min\Set{m'\in\mathbb{Z} | m'a+nb\in S(a,b)}\]
We know from the definition of $\mu_{a,b}(n)$ that 
\[\min\Set{m'\in\mathbb{Z} | m'a+(n+ia)b\in S(a,b)}=\mu_{a,b}(n)=\mu_{a,b}(n+ia)+ib\]
Therefore, $\mu_{a,b}(n+ia)=\mu_{a,b}(n)-ib$ for all $i,n\in\mathbb{Z}$.
\end{proof}

\section{The Lifting to $\mathbb{N}_0^2$ and the $\mu$ Sequence}

Now we will define a monoid in $\mathbb{N}_0^2$ that will allow us to unify all numerical semigroups generated by infinite quadratic sequences with initial term zero.

Let $T$ be the monoid generated by all linear combinations over $\mathbb{N}_0$ of the generators $z_i=\left(i, {i \choose 2}\right)$ where $i\in\mathbb{N}_0$.  That is,
\begin{align*}
    T&=\Set{\sum_{i=1}^{\infty} \lambda_iz_i | \lambda_i \in \mathbb{N}_0} \\
    &=\Set{\left(1\lambda_{1}+2\lambda_{2}+\ldots,{1 \choose 2}\lambda_{1}+{2 \choose 2}\lambda_{2}+\ldots\right) | \lambda_i \in \mathbb{N}_0}.
\end{align*}

The following theorem shows how the $T$ monoid is connected to our numerical semigroups $S(a,b)$.

\begin{theorem}\label{PhiFunction}
Let $S(a,b)\in\mathcal{Q}_0^\infty$, let $T$ be as previously defined, and let $\phi_{a,b}: \mathbb{N}_0^2 \to \mathbb{N}_0$ be given by $\phi_{a,b}(m,n)=ma+nb$.  Then $\phi_{a,b}[T]=S(a,b)$.
\end{theorem}
\begin{proof}
The image of $T$ under $\phi_{a,b}$ is
\begin{align*}
&\Set{ma+nb | (m,n) \in T} \\
&= \Set{ma+nb | m=1\lambda_{1}+2\lambda_{2}+\ldots, n={1 \choose 2}\lambda_{1}+{2 \choose 2}\lambda_{2}+\ldots, \lambda_i \in \mathbb{N}_0} \\
&= \Set{(1\lambda_{1}+2\lambda_{2}+\ldots)a+\left({1 \choose 2}\lambda_{1}+{2 \choose 2}\lambda_{2}+\ldots\right)b | \lambda_i \in \mathbb{N}_0} \\
&= \Set{\lambda_{1}\left(1a+{1 \choose 2}b\right)+\lambda_{2}\left(2a+{2 \choose 2}b\right)+\ldots | \lambda_i \in \mathbb{N}_0} \\
&= \Set{\lambda_{1}y_1+\lambda_{2}y_2+\ldots | \lambda_i \in \mathbb{N}_0}=S(a,b)
\end{align*}
\end{proof}

Hence, the monoid $T$ unifies all numerical semigroups $S(a,b)\in\mathcal{Q}_0^\infty$ in the sense that each such $S(a,b)$ is a particular projection of $T$.

\begin{figure}
  \centering
      \includegraphics[width=0.75\textwidth]{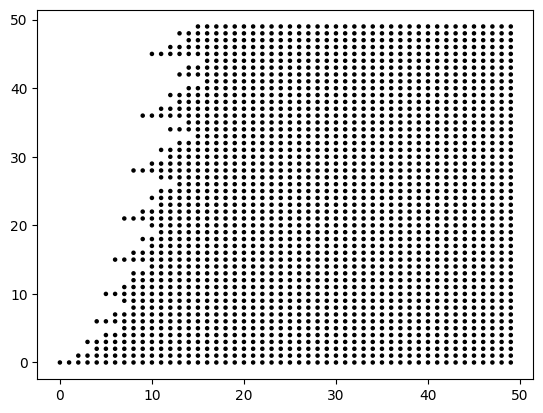}
  \caption{$\mathbb{N}_0^2$ with elements of $T$ colored black.}
   \label{Tpic}
\end{figure}

Now we will discuss the elements of $T$.  Figure \ref{Tpic} shows the elements $(m,n)\in T$ for $m,n<50$.  As can be seen in the figure, within each row, the color changes from white to black exactly once.  To say this more formally, for each value of $n\in \mathbb{N}_0$, there is some value $\mu(n)\in\mathbb{N}_0$ for which $m<\mu(n)$ implies $(m,n)\notin T$ and $m\geq \mu(n)$ implies $(m,n)\in T$.  This is so because $z_1=(1,0)$ is one of the generators of $T$, thus, for any $(m,n)\in T$, we have $(m+1,n)\in T$ as well.

This behavior is remarkably similar to that of an Ap\'ery set for a numerical semigroup, as well as to the behavior of $\mu_{a,b}$, and we will show in Section \ref{mu=mu_a,b section} that, in fact, $\mu_{a,b}(n)=\mu(n)$ when $0\leq n<a$, except for eight particular values of $(a,n)$.  Before we can show that, however, we need to establish many properties of $\mu$, in this and the next section.

We begin with a formal definition of $\mu(n)$:

\begin{definition}
Let $n\in\mathbb{N}_0$.  Then $\mu(n)$ is defined as
\[\mu(n):=\min\Set{m\in\mathbb{N}_0 | (m,n)\in T}\]
\end{definition}

Note that here, unlike with $\mu_{a,b}$, we have defined $\mu$ as a map on $\mathbb{N}_0$ rather than on $\mathbb{Z}$ because it is not possible for either $m$ or $n$ to be negative when $(m,n)\in T$.

Now we will establish some properties of $\mu$ that will allow us to prove an efficient method of calculating the $\mu$ values.

\begin{theorem}[Recursive bound on $\mu$]\label{recursiveBounds}
For all $S(a,b)\in\mathcal{Q}_0^\infty$ and for all $n_1,n_2\in\mathbb{N}_0$,
\[\mu(n_1+n_2)\leq\mu(n_1)+\mu(n_2)\]
\end{theorem}
\begin{proof}
Since $\mu(n_1)=\min\Set{m \in \mathbb{N}_0 | (m,(n_1)) \in T}$, that means $(\mu(n_1),n_1) \in T$. The same logic holds true for $\mu(n_2)$. Hence we can add together these two elements in $T$ to get another element that is in $T$. 
\[(\mu(n_1),n_1) + (\mu(n_2),n_2)=(\mu(n_1)+\mu(n_2),n_1+n_2) \in T.\] 
By definition, $\mu(n_1+n_2)=$min$\Set{m \in \mathbb{N}_0 | (m,n_1+n_2) \in T}$. Since $\mu(n_1)+\mu(n_2) \in \mathbb{N}_0$ and $(\mu(n_1)+\mu(n_2),n_1+n_2) \in T$,
\[\mu(n_1)+\mu(n_2) \in \Set{m \in \mathbb{N}_0 | (m,n_1+n_2) \in T}.\] 
Since $\mu(n_1)+\mu(n_2)$ is in the set we are taking the minimum over to get $\mu(n_1+n_2)$, we can say that
\[\mu(n_1+n_2)\leq\mu(n_1)+\mu(n_2).\]
\end{proof}

A particular application of the previous theorem is the following.

\begin{corollary}\label{recursiveBound2}
For all $n,i\in\mathbb{N}_0$, if ${i \choose 2}\leq n$, then
\[\mu(n)\leq \mu\left(n-{i \choose 2}\right)+i\]
\end{corollary}
\begin{proof}
Assume $i\in\mathbb{N}_0$ and ${i\choose 2}\leq n$.  The previous theorem states that if $n_1,n_2\in\mathbb{N}_0$, then $\mu(n_1+n_2) \leq \mu(n_1)+ \mu(n_2)$. Let $n_1=n-{i \choose 2}$ and $n_2={i \choose 2}$. By the assumptions about $i$, we can see that $n_1,n_2\in\mathbb{N}_0$.  By plugging these values into the inequality, we get 
\[\mu\left(n-{i \choose 2}+{i \choose 2}\right) \leq \mu\left(n-{i \choose 2}\right)+ \mu\left({i \choose 2}\right)\] 
Simplifying, we obtain 
\[\mu\left(n\right) \leq \mu\left(n-{i \choose 2}\right)+\mu\left({i \choose 2}\right)\]
We know that $z_i=\left(i,{i\choose 2}\right)\in T$. So, by the definition of $\mu(i)$, it must be that $\mu\left({i\choose 2}\right)\leq i$. Therefore, 
\[\mu\left(n\right) \leq \mu\left(n-{i \choose 2}\right)+i\]
\end{proof}

The next theorem provides a computationally efficient way of calculating $\mu$ values.  Note that the theorem statement refers to $\mathbb{N}$, not $\mathbb{N}_0$. Values of $\mu$ that were calculated using this theorem, as well as a Python implementation of the theorem, can be found at \cite{OEIS}.

\begin{theorem}\label{recursiveExact}
For all $n\in\mathbb{N}$,
\[\mu(n)=\min\Set{\mu\left(n-{i \choose 2}\right)+i | i\in\mathbb{N}, {i \choose 2}\leq n}\]
\end{theorem}
\begin{proof}
Let $n\in\mathbb{N}$.  We know that $\mu(n)$ is defined as the minimum value of $1c_1+2c_2+\ldots$ such that $c_1,c_2,\ldots \in \mathbb{N}_0$ and $n={1\choose2}c_1+{2\choose2}c_2+\ldots$. So, there exist some $c_1,c_2,\ldots c_k \in \mathbb{N}_0$ where $k\in\mathbb{N}$ such that 
\[\mu(n)=1c_1+2c_2+\ldots+kc_k\] 
\[n={1\choose2}c_1+{2\choose2}c_2+\ldots+{k\choose2}c_k\]
We know that ${k\choose 2}\leq n$, because if ${k\choose 2}>n$, then the right hand side of the second equation from above will be greater than the left hand side. Also, since $n\neq 0$, then there exists some $i\in\{1,2,\ldots,k\}$ such that $c_i> 0$, because all $c_i\geq 0$ and not all of them can be 0 because $n\neq 0$. By subtracting $i$ from both sides of our $\mu(n)$ equation, we obtain 
\begin{align*}
    \mu(n)-i
    &=-i+1c_1+2c_2+\ldots+kc_k \\
    &=1c_1+2c_2+\ldots+(i-1)c_{i-1}+i(c_i-1)+(i+1)c_{i+1}+\ldots+kc_k
\end{align*}
Since $c_i > 0,$ $c_i-1\in\mathbb{N}_0.$
Next, we can subtract ${i\choose2}$ from both sides of our $n$ equation to obtain
\begin{align*}
   n-{i\choose2}
   &=-{i\choose2}+{1\choose2}c_1+{2\choose2}c_2+\ldots+{k\choose2}c_k \\
   &={1\choose2}c_1+{2\choose2}c_2+\ldots+{i-1\choose 2}c_{i-1}+{i\choose 2}(c_1-1)+ \\
   &\;\;\;\;\; {i+1 \choose 2}c_{i+1}+\ldots+{k\choose2}c_k
\end{align*}
Since $c_i> 0$, $c_i-1\in\mathbb{N}_0$. Therefore, all of our coefficients will still be in $\mathbb{N}_0$, so 
\[\left(\mu(n)-i, n-{i\choose2}\right) \in T.\] 
So then, by the definition of $\mu\left(n-{i\choose2}\right)$,
\[\mu(n)-i\geq \mu\left(n-{i\choose2}\right)\] 
which implies that 
\[\mu(n)\geq \mu\left(n-{i\choose2}\right)+i.\] 
However, Corollary \ref{recursiveBound2} states that for all $n,i\in\mathbb{N}_0$, if ${i \choose 2}\leq n$, then
\[\mu(n) \leq \mu\left(n-{i\choose2}\right)+i\] 
Therefore, $\mu(n)=\mu\left(n-{i \choose 2}\right)+i$ for some $i\in\mathbb{N}$ such that ${i \choose 2}\leq n$, while at the same time $\mu(n)\leq \mu\left(n-{i \choose 2}\right)+i$ for all $i\in\mathbb{N}_0$. So, it must be the case that 
\[\mu(n)=\min\Set{\mu\left(n-{i \choose 2}\right)+i | i\in\mathbb{N}, {i \choose 2}\leq n}.\]
\end{proof}

\section{Bounding the $\mu$ Sequence}

As noted previously, we will show in Section \ref{mu=mu_a,b section} that, in fact, $\mu_{a,b}(n)=\mu(n)$ when $0\leq n<a$, except for eight particular values of $(a,n)$.  This implies that the Ap\'ery set, the Frobenius number, and the genus of $S(a,b)\in\mathcal{Q}_0^\infty$ can be written purely in terms of $\mu$ rather than $\mu_{a,b}$.  However, in order to prove that, we need more information than we presently have about the values of $\mu$.

In addition to that upcoming application of $\mu(n)$, the $\mu(n)$ sequence is of some interest in its own right, as it is related to integer partitions, so it is a worthwhile exercise to investigate its values.

From the definition of $\mu(n)$, we can see that each $\mu(n)$ is the optimal solution of an integer linear program:

\[\mu(n)=\min\Set{1\lambda_{1}+2\lambda_{2}+\ldots |  \lambda_1,\lambda_2,\ldots\in\mathbb{N}_0, n={1 \choose 2}\lambda_{1}+{2 \choose 2}\lambda_{2}+\ldots}.\]

Finding the optimum value of an integer linear program is known to be NP-hard \cite{papadimitriou81}, so it is not reasonable to expect that we can find a closed formula for $\mu(n)$ as a function of $n$.  We only know of one case in which a closed formula is known, which is shown in Corollary \ref{muOfIChoose2}.  Thus, we turn our attention now to developing upper and lower bounds for $\mu$.

First, let us  define and look at the properties of a function that will be of much use throughout the remainder of this section and the next.  This function is the inverse of the $x \choose 2$ function for $x\in\mathbb{N}$.

\begin{definition}
The function $f:\mathbb{N}\to\mathbb{R}$ is given by $f(x)=\frac{1+\sqrt{8x+1}}{2}$.
\end{definition}

\begin{lemma}\label{flemma}
If $x\in\mathbb{N}$, then $f\left({x \choose 2}\right)=x$ and ${f(x) \choose 2}=x$.
\end{lemma}
\begin{proof}
First we will show that if $x\geq 1$, then $f\left({x \choose 2}\right)=x$. Since ${ x \choose 2}=\frac{x(x-1)}{2}$, $f\left({ x \choose 2}\right)=f\left(\frac{x(x-1)}{2}\right)$.
\begin{align*}
f\left(\frac{x(x-1)}{2}\right)&=\frac{1+\sqrt{8\left(\frac{x(x-1)}{2}\right)+1}}{2} =\frac{1+\sqrt{4x^2-4x+1}}{2} \\
&=\frac{1+\sqrt{(2x-1)^2}}{2} =\frac{1+|2x-1|}{2}=\frac{1+2x-1}{2}  =\frac{2x}{2}=x
\end{align*}
Next we will show that if $x\geq 1$, then ${f(x) \choose 2}=x$. We know that ${f(x) \choose 2}=\frac{(f(x))(f(x)-1)}{2}$. So,
\begin{align*}
\frac{1}{2}(f(x))(f(x)-1)&=\frac{1}{2}\left(\frac{1+\sqrt{8x+1}}{2}\right)\left(\frac{1+\sqrt{8x+1}}{2}-1\right) \\
&=\frac{1}{2}\left(\frac{1+\sqrt{8x+1}}{2}\right)\left(\frac{-1+\sqrt{8x+1}}{2}\right) \\
&=\frac{-1+8x+1}{8} =\frac{8x}{8}=x
\end{align*}
\end{proof}

With the $f$ function at our disposal, a lower bound on $\mu$ is quite straightforward to find, and this bound is in fact tight for infinitely many values of $n$, as will be shown in Corollary \ref{muOfIChoose2}.

\begin{theorem}\label{fBound}
For all $n\in\mathbb{N}$, $\mu(n)\geq f(n)$.
\end{theorem}
\begin{proof}
Assume there is some $n\in \mathbb{N}_0$ such that $\mu(n)<f(n)$. Let $\mu(n)=m$. Then since $(m,n)\in T$, from the definition of $T$, there must be some $c_1,c_2,\ldots,c_k$ for $k\in\mathbb{N}_0$ such that 
\[m=1c_1+2c_2+\ldots+kc_k\] and \[n=c_1{1 \choose 2}+c_2{2\choose 2}+\ldots+c_k{k \choose 2}.\]

Since $m<f(n)$, and the function ${x\choose 2}$ is increasing when $x\geq 1$, and $m,f(n)\geq 1$, we can apply this function to both sides of $m<f(n)$ to get 
${m\choose 2}<n$. From our $m$ and $n$ equations, we can write ${m \choose 2}<n$ as \[{1c_1+2c_2+\ldots+kc_k \choose 2}<c_1{1\choose2}+c_2{2\choose 2}+\ldots+c_k{k \choose 2}.\] However, the ${x \choose 2}$ function is convex, so it is superadditive, which means that that ${x+y \choose 2} \geq {x \choose 2}+{y \choose 2}$ for all $x,y\in\mathbb{R}$, so \[{1c_1+2c_2+\ldots+kc_k \choose 2} \geq c_1{1 \choose 2}+ c_2{2\choose 2}+\ldots+c_k{k\choose 2}.\] So we have a contradiction, meaning our assumption was false. So $\mu(n) \geq f(n)$.
\
\end{proof}

In fact, the bound in the previous theorem is tight for infinitely many values of $n$, due to the following corollary.  This is the only infinite family of values $n$ for which we can calculate the exact value of $\mu(n)$ without resorting to recursion.

\begin{corollary}\label{muOfIChoose2}
For all $i\in\mathbb{N}$ where $i\geq 2$, $\mu\left(i\choose 2\right)=i$.
\end{corollary}
\begin{proof}
Let $i\in\mathbb{N}$ and $i\geq 2$.  Then ${i\choose 2}\in\mathbb{N}$.  Theorem \ref{fBound} states that for all $n \in \mathbb{N}$, $\mu(n) \geq f(n)$. So, $\mu\left({i \choose 2}\right)\geq f\left({i \choose 2}\right)$. Also, according to Lemma \ref{flemma}, $f\left({i \choose 2}\right)=i$. Therefore, $\mu\left({i \choose 2}\right) \geq i$. However, Corollary \ref{recursiveBound2} states that for all $n,i\in\mathbb{N}_0$, if ${i \choose 2}\leq n$, then
\[\mu(n)\leq \mu\left(n-{i \choose 2}\right)+i\]
So, when $n={i\choose 2}$, 
\[\mu\left({i \choose 2}\right)\leq \mu\left({i \choose 2}-{i \choose 2}\right)+i\]
Simplifying, we obtain
\[\mu\left({i \choose 2}\right)\leq \mu(0)+i\]
Since $\mu(0)=0$, we have $\mu\left({i \choose 2}\right)\leq i$.  Therefore, we can conclude that $\mu({i\choose 2})=i$.
\end{proof}

A tight upper bound is far more difficult to find.  We will begin by proving what we refer to as the Gauss bound, because we will use the so-called ``Eureka'' theorem of Gauss (see \cite{nathanson87} for a modern treatment) to prove it.  This is a celebrated result of Gauss which says that any natural number can be written as the sum of three triangular numbers.  Triangular numbers are those numbers which are equal to $i \choose 2$ for some $i\in\mathbb{N}_0$.

Although this bound is not tight, it is closed and non-recursive, so it is easy to work with, and it is sufficient to prove various useful results that appear in the following sections.

\begin{theorem}[Gauss bound]\label{GaussBound}
For all $n\in\mathbb{N}_0$, $\mu(n)\leq 3f\left(\frac{n}{3}\right)$.
\end{theorem}
\begin{proof}
By definition,
\[\mu(n)=\min\Set{1\lambda_{1}+2\lambda_{2}+\ldots |  \lambda_1,\lambda_2,\ldots\in\mathbb{N}_0, n={1 \choose 2}\lambda_{1}+{2 \choose 2}\lambda_{2}+\ldots}\]

Hence, if there exist $\lambda_1,\lambda_2,\ldots\in\mathbb{N}_0$ such that $n={1 \choose 2}\lambda_1+{2 \choose 2}\lambda_2+\dots$, we have $\mu(n)\leq 1\lambda_1+2\lambda_2+\dots$.  Gauss' Eureka theorem states that there exist $x,y,z\in\mathbb{N}_0$ such that $n={x \choose 2}+{y \choose 2}+{z \choose 2}$, hence

\begin{align*}
\mu(n)
&\leq\min\Set{x+y+z |  x,y,z\in\mathbb{N}_0, n={x \choose 2}+{y \choose 2}+{z \choose 2}} \\
&\leq\max\Set{x+y+z |  x,y,z\in\mathbb{N}_0, n={x \choose 2}+{y \choose 2}+{z \choose 2}} \\
&\leq\max\Set{x+y+z |  x,y,z\in\mathbb{R}, n={x \choose 2}+{y \choose 2}+{z \choose 2}} \\
\end{align*}

We now simply need to find the maximum of $x+y+z$ over the reals, subject to the constraint $n={x \choose 2}+{y \choose 2}+{z \choose 2}$, which is a straightforward optimization problem that can solved with analysis.

We can make this optimization problem easier by noting that the constraint $n={x \choose 2}+{y \choose 2}+{z \choose 2}$ is a sphere.  The level sets of the objective function $x+y+z$, on the other hand, are planes.  Hence, the minimum and maximum values of the objective function will occur at the two level sets of $x+y+z$ whose planes are tangent to the sphere.

The two points at which tangency occurs are when $x=y=z=f(n/3)$, which yields the maximum value of $x+y+z=3f(n/3)$ and $x=y=z=1-f(n/3)$, which yields the minimum value of $x+y+z=3-3f(n/3)$.  Hence $\mu(n)\leq 3f(n/3)$.

\end{proof}

The advantage of the Gauss bound is that the formula is closed and easy to write.  The disadvantage is that the bound does not seem to be very good---in fact, we have not encountered any value of $n$ for which the Gauss bound is tight.  Numerical evidence for the looseness of the Gauss bound is shown in Figure \ref{bounds}.

\begin{figure}
  \centering
      \includegraphics[width=1\textwidth]{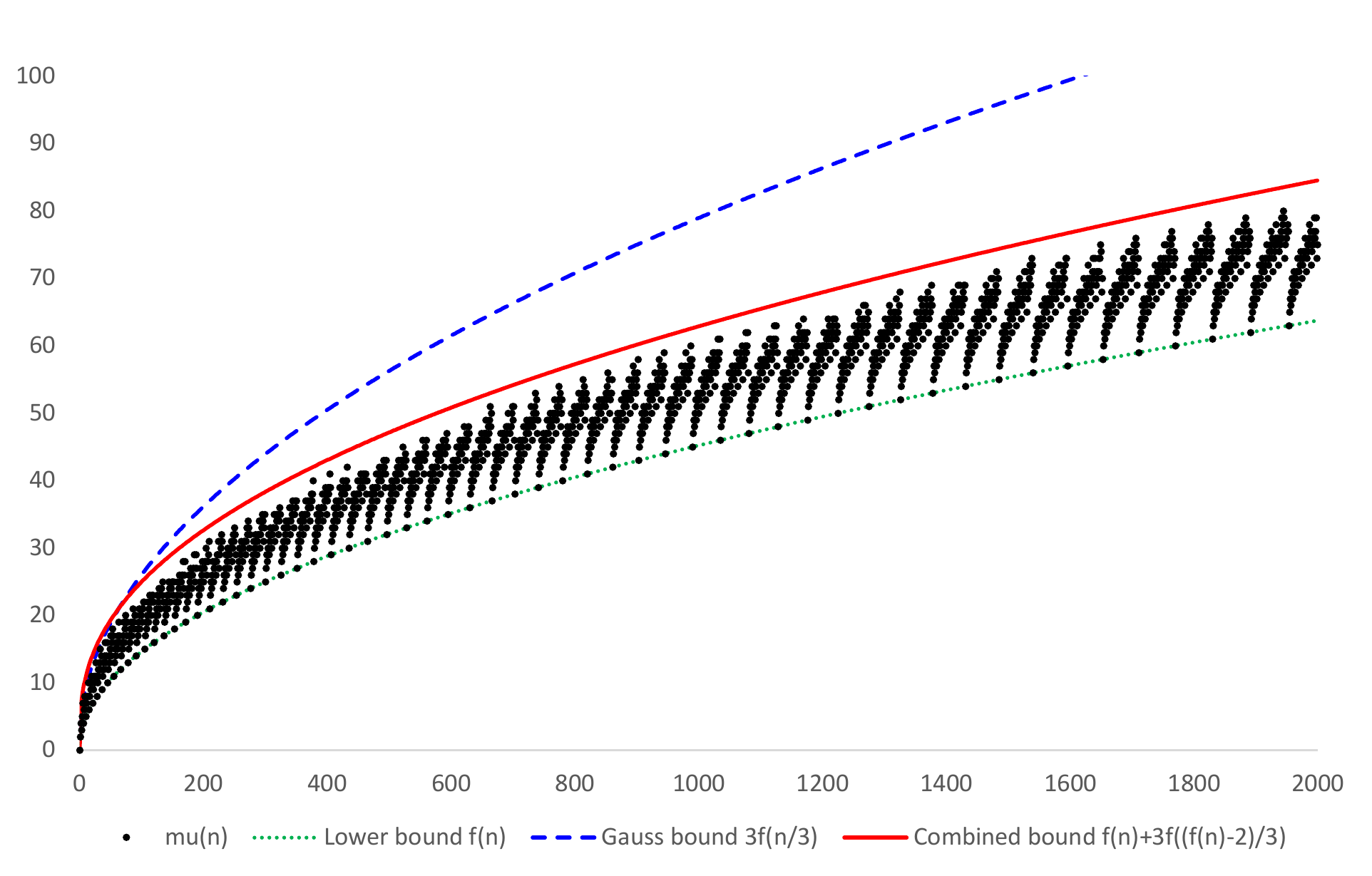}
  \caption{$\mu(n)$ for $n<2000$, along with various upper and lower bounds.}
   \label{bounds}
\end{figure}

The upper recursive bound given in Corollary \ref{recursiveBound2}, on the other hand, seems to be much tighter; however, the difficulty in using the recursive upper bound arises from the fact that it is recursive.  We can, for example, choose the largest value of $i\in\mathbb{N}_0$ such that ${i \choose 2}\leq n$, which is $i=\lfloor f(n)\rfloor$, and then use the recursive bound to say that
\[\mu(n)\leq \lfloor f(n)\rfloor + \mu\left(n-{\lfloor f(n)\rfloor \choose 2}\right)\]

However, we now have to bound the second term, which is another value of $\mu$, using either the Gauss bound or another instance of the recursive upper bound.  If we use the recursive upper bound repeatedly, the formula becomes increasingly unwieldy and difficult to understand.  Applying the recursive bound repeatedly is essentially similar to whittling away at the argument of $\mu$, subtracting as large a chunk as possible at each application of the recursive bound.  It is also difficult to know how many applications of the recursive upper bound are necessary before the argument is whittled down to zero.

In spite of these difficulties, we did pursue this approach at some length.  Unfortunately, the formula for the upper bound that results from exhaustive applications of the recursive upper bound is also recursive, as we have just demonstrated.  Furthermore, the formula is so complicated and difficult to calculate that, from a human perspective, one might as well just compute the exact values of $\mu$ using the recursive formula in Theorem \ref{recursiveExact}.  The only possible advantage to computing the bound rather than $\mu$ itself would be that the time complexity of computing the bound is of a lower order than that of computing $\mu$, so from a computational perspective, it might have some benefit.  Furthermore, all our attempts to express the resulting bound in a closed form by loosening it resulted in worse bounds than the Gauss bound.

Nevertheless, we will give one final bound in this section, which is the result of combining one application of the recursive bound with the Gauss bound.  The formula for this bound is complicated, but it is not recursive, and it is much tighter than the Gauss bound, as shown in Figure \ref{bounds}.

\begin{theorem}[Combined bound]\label{combined bound}
For all $n\in\mathbb{N}$,
\[\mu(n)\leq f(n)+3f\left(\frac{f(n)-2}{3}\right)\]
\end{theorem}
\begin{proof}
Let $n\in\mathbb{N}$.  We know from Lemma \ref{flemma} that for $n\in\mathbb{N}$, ${f(n) \choose 2}=n$.  Since $x \choose 2$ is an increasing function for $x\geq 1$ and $f(n)\geq 1$ for $n\in\mathbb{N}_0$, this implies ${\lfloor f(n)\rfloor \choose 2}\leq {f(n) \choose 2}$.  Hence ${\lfloor f(n)\rfloor \choose 2}\leq n$, and so we can apply the recursive bound (Theorem \ref{recursiveBounds}) with $i=\lfloor f(n)\rfloor$.  Doing so yields
\[\mu(n)\leq\lfloor f(n)\rfloor+\mu\left(n-{\lfloor f(n)\rfloor\choose 2}\right)\]
We can then drop the floor and apply the Gauss bound to the second term to get
\[\mu(n)\leq f(n)+ 3f\left(\frac{n-{\lfloor f(n)\rfloor\choose 2}}{3}\right)\]

Now we would like to get rid of the remaining floor function, starting with the observation that $\lfloor f(n)\rfloor>f(n)-1$. Since $f(n)\geq 2$ when $n\geq 1$, both sides of the inequality $\lfloor f(n)\rfloor>f(n)-1$ are greater than or equal to $1$.  Since the $x\choose 2$ function is increasing for $x\geq 1$, we can apply the $x\choose 2$ function to both sides of $\lfloor f(n)\rfloor>f(n)-1$ to get
\begin{align*}
    {\lfloor f(n)\rfloor\choose 2}>{f(n)-1\choose 2}&=\frac{1}{2}(f(n)-1)(f(n)-2) \\
    &=\frac{1}{2}(f(n)-1)f(n)-\frac{1}{2}(f(n)-1)(2) \\
    &={f(n)\choose 2}-(f(n)-1)\\
    &=n-(f(n)-1)
\end{align*}
Hence
\[n-{\lfloor f(n)\rfloor\choose 2} < n-(n-(f(n)-1))=f(n)-1\]
Since $n-{\lfloor f(n)\rfloor\choose 2}$ is an integer, we can tighten this to
\[n-{\lfloor f(n)\rfloor\choose 2} \leq f(n)-2\]

Since $f$ is an increasing function, this yields
\[\mu(n)\leq f(n)+3f\left(\frac{f(n)-2}{3}\right)\]
\end{proof}

\section{The Relationship between $\mu$ and $\mu_{a,b}$}\label{mu=mu_a,b section}

At this point in the paper, we have now established enough theorems to show the oft-mentioned fact that $\mu(n)=\mu_{a,b}(n)$ when $0\leq n<a$, except for eight particular values of $(a,n)$.  However, the proof of this theorem is quite long and involved, so, in order to make it easier to read, we have broken off some portions of the proof and made them into lemmas.  These lemmas are quite technical, and probably of no particular interest apart from supporting the proof of the theorem.

\begin{lemma}\label{mu=mu(a,b) small cases}
For all $S(a,b)\in\mathcal{Q}_0^\infty$ with $a,b\geq 1$,
\begin{enumerate}
    \item $\mu_{a,b}(0)=0=\mu(0)$
    \item If $a\geq 2$, then $\mu_{a,b}(1)=2=\mu(1)$
    \item If $a\geq 3$, then $\mu_{a,b}(2)=4=\mu(2)$
\end{enumerate}
\end{lemma}
\begin{proof}
Let $S(a,b)\in\mathcal{Q}_0^\infty$ with $a,b\geq 1$. We know from the definition of $\mu_{a,b}(n)$ that $\mu_{a,b}(0)=\min\Set{m\in\mathbb{Z} | ma+(0)b\in S(a,b)}$. Since $0$ is the smallest element of $S(a,b)$, $\mu_{a,b}(0)$ is clearly $0$.

Now, assume $a\geq 2$. Since $2a+b\in S(a,b)$, we know that $\mu_{a,b}(1)\leq 2$. Suppose that $\mu_{a,b}(1)<2$. Then there is some $m\in\mathbb{Z}$ with $m\leq 1$ and $ma+b\in S(a,b)$. Then $ma+b\leq a+b$. The only generator less than or equal to $a+b$ is $a$, so $ma+b$ must be a multiple of $a$. So there is some $i\in\mathbb{N}$ such that $ma+b=ia$, which implies $b=(i-m)a$. However, since $a$ divides the right hand side, $a$ must also divide $b$, which is impossible. Thus $\mu_{a,b}(1)=2$.

Now, assume $a\geq 3$. Since $2(2a+b)=4a+2b\in S(a,b)$, we know that $\mu_{a,b}(2)\leq 4$. Suppose that $\mu_{a,b}(2)<4$. Then there is some $m\in\mathbb{Z}$ with $m\leq 3$ and $ma+2b\in S(a,b)$. Then $ma+2b\leq 3a+2b$. The only generators less than or equal to $3a+2b$ are $a$ and $2a+b$. So there are some $i,j\in\mathbb{N}$ such that $ma+2b=ia+j(2a+b)$. Furthermore, since $ma+2b\leq 3a+2b$, we must have $j\leq 1$, because $j\geq 2$ makes $ia+j(2a+b)\geq ia+4a+2b$, which is strictly greater than $ma+2b$. Rearranging to collect all copies of $a$ and $b$, we get $(2-j)b=(i+2j-m)a$. Since $a$ divides the right hand side, it must also divide the left hand side, and since $\gcd(a,b)=1$, $a$ must divide $2-j$. Since $0\leq j\leq 1$, that means $1\leq 2-j\leq 2$. This is impossible, however, because $a\geq 3$, so it cannot divide either 1 or 2. Hence $\mu_{a,b}(2)=4$.

The fact that $\mu(0)=0$, $\mu(1)=2$, and $\mu(2)=4$ is due to Theorem \ref{recursiveExact}.
\end{proof}

\begin{lemma}\label{mu(a,b)<=mu}
For all $S(a,b)\in\mathcal{Q}_0^\infty$ with $a,b\geq 1$ and all $n\in\mathbb{N}_0$, $\mu_{a,b}(n)\leq\mu(n)$.
\end{lemma}
\begin{proof}
Let $S(a,b)\in\mathcal{Q}_0^\infty$ with $a,b\geq 1$ and $n\in\mathbb{N}_0$. Also, let $(\mu(n),n) \in T$ such that $\mu(n)=\min\Set{m\in\mathbb{N} | (m,n)\in T}$, and let $\mu_{a,b}(n)a+nb \in S(a,b)$ such that $\mu_{a,b}(n)=\min\Set{m\in\mathbb{Z} | ma+nb\in S(a,b)}$. 

Suppose $\mu_{a,b}(n)>\mu(n)$. Since we defined $\phi_{a,b}: \mathbb{N}_0^2 \to \mathbb{N}_0$ as $\phi_{a,b}(m,n)=ma+nb$ and Theorem \ref{PhiFunction} states that the image of $T$ under $\phi_{a,b}$ is $S(a,b)$, $\phi_{a,b}(\mu(n),n)=\mu(n)a+nb \in S$. We assumed that $\mu_{a,b}(n)a+nb \in S(a,b)$ such that $\mu_{a,b}(n)=\min\Set{m\in\mathbb{Z} | ma+nb\in S(a,b)}$, but $\mu(n)<\mu_{a,b}(n)$ and $\mu(n) \in \mathbb{Z}$ and $\mu(n)a+nb \in S$, so this is a contradiction. Therefore, for all $n\in\mathbb{N}_0$, $\mu_{a,b}(n)\leq\mu(n)$.
\end{proof}

\begin{lemma}\label{mu(a,b)(n+a)=mu(n+a)}
For all $S(a,b)\in\mathcal{Q}_0^\infty$ with $a,b\geq 1$, if there exists $n\in\mathbb{N}_0$ such that $n<a$ and $\mu_{a,b}(n)<\mu(n)$, then $\mu_{a,b}(n+a)=\mu(n+a)$.
\end{lemma}
\begin{proof}
Let $S(a,b)\in\mathcal{Q}_0^\infty$ with $a,b\geq 1$, and suppose there exists $n\in\mathbb{N}_0$ such that $n<a$ and $\mu_{a,b}(n)<\mu(n)$.  Due to Lemma \ref{mu=mu(a,b) small cases}, we know that $n\geq 3$.

By the definition of $\mu_{a,b}$, we know that $\mu_{a,b}(n)a+nb\in S(a,b)$.  Then, by Theorem \ref{PhiFunction}, there exists $(m',n')\in T$ such that $\mu_{a,b}(n)a+nb=m'a+n'b$.  We can rearrange this equation to get $(\mu_{a,b}(n)-m')a=(n'-n)b$.  Since $\gcd(a,b)=1$, this implies that there exists $k\in\mathbb{Z}$ such that $\mu_{a,b}(n)-m'=kb$ and $n'-n=ka$.

Suppose that $k=0$.  Then $(\mu_{a,b}(n),n)=(m',n')$, so $(\mu_{a,b}(n),n)\in T$.  Then, by the definition of $\mu$, $\mu(n)\leq \mu_{a,b}(n)$, which contradicts the assumption that $\mu_{a,b}(n)<\mu(n)$.  Now suppose that $k<0$.  Then $n'-n=ka\leq -a$, so $n'\leq n-a$.  However, since $n<a$, this implies $n'<0$, which contradicts the fact that $(m',n')\in T$, which is a subset of $\mathbb{N}_0^2$.  Hence, it must be the case that $k\geq 1$.

Since $(m',n')\in T$, by the definition of $\mu$, $\mu(n')\leq m'$.  Thus, $\mu(n+ka)\leq \mu_{a,b}(n)-kb$.  Using Theorem \ref{mu(a,b) of n+ia}, this becomes $\mu(n+ka)\leq \mu_{a,b}(n+ka)$ which is the same as $\mu(n')\leq\mu_{a,b}(n')$.  On the other hand, Lemma \ref{mu(a,b)<=mu} tells us that $\mu_{a,b}(n')\leq\mu(n')$.  Hence, $\mu_{a,b}(n')=\mu(n')$, which is the same as $\mu_{a,b}(n+ka)=\mu(n+ka)$.

Now we will show that $k=1$, which will yield the statement we are trying to prove.

Theorem \ref{mu(a,b) of n+ia} tells us that $\mu_{a,b}(n)=\mu_{a,b}(n+ka)+kb$, so $\mu_{a,b}(n)=\mu(n+ka)+kb$.  Furthermore, since we assumed that $\mu_{a,b}(n)<\mu(n)$, there exists some $\ell\in\mathbb{N}$ such that $\mu_{a,b}(n)=\mu(n)-\ell$.  Hence, $\mu(n)-\ell= \mu(n+ka)+kb$.  Rearranged, this is $\mu(n+ka)+kb+\ell=\mu(n)$.

Now we will apply the lower bound on $\mu$ from Theorem \ref{fBound} to the left-hand side of the equation, and we will apply the upper bound on $\mu$ from Theorem \ref{GaussBound} to the right-hand side of this equation, to obtain
\begin{align*}
    f\left(n+ka\right)+kb+\ell &\leq 3f\left(\frac{n}{3}\right) \\
    \frac{1}{3}\left(f\left(n+ka\right)+kb+\ell\right) &\leq f\left(\frac{n}{3}\right) \\
\end{align*}

Since $x\geq 0$ implies $f(x)\geq 1$, and since $k,b,\ell\geq 1$, the left-hand side of this inequality is at least $1$, as is the right-hand side.  The $x\choose 2$ function is increasing when $x\geq 1$, so we can apply the $x\choose 2$ function to both sides of this inequality.
\[    {\frac{1}{3}\left(f(n+ka)+kb+\ell\right) \choose 2} \leq {f\left(\frac{n}{3}\right)\choose 2}\]

Furthermore, since we know $n\geq 3$, we have $\frac{n}{3}\geq 1$, so we can apply Lemma \ref{flemma} to say that
\[{f\left(\frac{n}{3}\right)\choose 2}=\frac{n}{3}\]
Hence
\begin{align*}
    \frac{1}{2}\left(\frac{1}{3}\left(f(n+ka)+kb+\ell\right)\right)\left(\frac{1}{3}\left(f(n+ka)+kb+\ell\right)-1\right) &\leq \frac{n}{3} \\
    \frac{1}{18}\left(f(n+ka)+kb+\ell\right)\left(f(n+ka)+kb+\ell-3\right) &\leq \frac{n}{3} \\
    \left(f(n+ka)+kb+\ell\right)\left(f(n+ka)+kb+\ell-3\right) &\leq 6n \\
    f(n+ak)^2+(2kb+2\ell-3)f(n+ak)+(kb+\ell)(kb+\ell-3) &\leq 6n \\
    f(n+ak)^2-f(n+ak)+(2kb+2\ell-2)f(n+ak)+(kb+\ell)(kb+\ell-3) &\leq 6n \\
    2{f(n+ak)\choose 2}+(2kb+2\ell-2)f(n+ak)+(kb+\ell)(kb+\ell-3) &\leq 6n
\end{align*}

Since $a,k\geq 1$, we have $n+ak\geq 1$, so we can apply Lemma \ref{flemma} to say that
\[{f(n+ak)\choose 2}=n+ak\]
Thus the inequality becomes
\[2(n+ak)+(2kb+2\ell-2)f(n+ak)+(kb+\ell)(kb+\ell-3) \leq 6n\]
\begin{equation}\label{ineq1}
    (2kb+2\ell-2)f(n+ak)+(kb+\ell)(kb+\ell-3) \leq 4n-2ak
\end{equation}

Now let us consider the quantities on the left-hand side of the inequality.  Since $k,b,\ell\geq 1$, this implies $2kb+2\ell-2\geq 2$.  We also have $f(x)\geq 1$ whenever $x\geq 0$, so $(2kb+2\ell-2)f(n+ak)\geq 2$.  As for $(kb+\ell)(kb+\ell-3)$, when $kb+\ell\geq 3$, we have $(kb+\ell)(kb+\ell-3)\geq 0$.  Otherwise, if $kb+\ell< 3$, then since $k,b,\ell\geq 1$, we must have $kb+\ell=2$, in which case $(kb+\ell)(kb+\ell-3)=-2$.  Taken together, all of these imply that $0\leq 4n-2ak$.  Rearranged, this is $ak\leq 2n$.

Furthermore, since we assumed that $n<a$, $ak\leq 2n$ implies that $ak<2a$, so $k<2$.  Since we already know that $k\geq 1$, it must be the case that $k=1$.  Hence, $\mu_{a,b}(n+a)=\mu(n+a)$.
\end{proof}

\begin{lemma}\label{a<=485 lemma}
For all $S(a,b)\in\mathcal{Q}_0^\infty$ with $a,b\geq 1$, if there exists $n\in\mathbb{N}_0$ such that $n<a$ and $\mu_{a,b}(n)<\mu(n)$, then $a\leq 485$ and $b+\mu(n)-\mu_{a,b}(n)\leq 4$.
\end{lemma}
\begin{proof}
Let $S(a,b)\in\mathcal{Q}_0^\infty$ with $a,b\geq 1$, and assume there exists $n\in\mathbb{N}_0$ such that $n<a$ and $\mu_{a,b}(n)<\mu(n)$.  Due to Lemma \ref{mu=mu(a,b) small cases}, we know $n\geq 3$, hence we have $a\geq 4$ as well.  From the last lemma, we know that $\mu(n+a)=\mu_{a,b}(n+a)$, and from Lemma \ref{mu(a,b) of n+ia}, this tells us $\mu(n+a)=\mu_{a,b}(n)-b$.

Let $\ell$ stand for the amount by which $\mu(n)$ exceeds $\mu_{a,b}(n)$, in other words, $\ell:=\mu(n)-\mu_{a,b}(n)$.  Then $\mu(n+a)=\mu_{a,b}(n)-b=\mu(n)-\ell-b$, so we have $\mu(n+a)+b+\ell=\mu(n)$.

Now we will apply the lower bound from Theorem \ref{fBound} to the left side of the equation and the combined bound from Theorem \ref{combined bound} to the right side of the equation to get
\[f(n+a)+b+\ell\leq f(n)+ 3f\left(\frac{f(n)-2}{3}\right)\]
Rearranging, we get
\[b+\ell\leq f(n)-f(n+a)+ 3f\left(\frac{f(n)-2}{3}\right)\]

Now we will show that the right-hand side of the last inequality is increasing on $n$.  Clearly, the last term, $3f\left((f(n)-2)/3\right)$, is increasing on $n$, because $f$ is an increasing function, so we just need to show that $f(n)-f(n+a)$ is increasing on $n$ as well.
\begin{align*}
    f(n)&-f(n+a)=\frac{1}{2}\left(\sqrt{8n+1}-\sqrt{8(n+a)+1}\right)\\
    &=\frac{1}{2}\left(\sqrt{8n+1}-\sqrt{8(n+a)+1}\right)\left(\frac{\sqrt{8n+1}+\sqrt{8(n+a)+1}}{\sqrt{8n+1}+\sqrt{8(n+a)+1}}\right)\\
    &=\frac{1}{2}\left(\frac{8n+1-(8(n+a)+1)}{\sqrt{8n+1}+\sqrt{8(n+a)+1}}\right)\\
    &=\frac{-4a}{\sqrt{8n+1}+\sqrt{8(n+a)+1}}
\end{align*}
Clearly, the denominator of this expression is increasing in $n$, and since the numerator is a negative constant, the overall expression is also increasing in $n$.

Since $f(n)-f(n+a)+ 3f\left(\frac{f(n)-2}{3}\right)$ increases with $n$, for $n<a$, it will achieve its maximal value when $n=a-1$.  Hence
\[b+\ell\leq f(a-1)-f(2a-1)+ 3f\left(\frac{f(a-1)-2}{3}\right)\]

Let us define the right-hand side of the inequality as a function of $a$.  Let $g: \mathbb{N}\to \mathbb{R}$ be given by
\[g(a)=f(a-1)-f(2a-1)+ 3f\left(\frac{f(a-1)-2}{3}\right)\]

This function is not easy to work with algebraically, but it is smooth and continuous, so we can find some of its properties by examining the graph, shown in Figure \ref{gPic}.  For example, there is a local maximum at approximately $g(52.15)\approx 4.59$.

\begin{figure}
  \centering
      \includegraphics[width=1\textwidth]{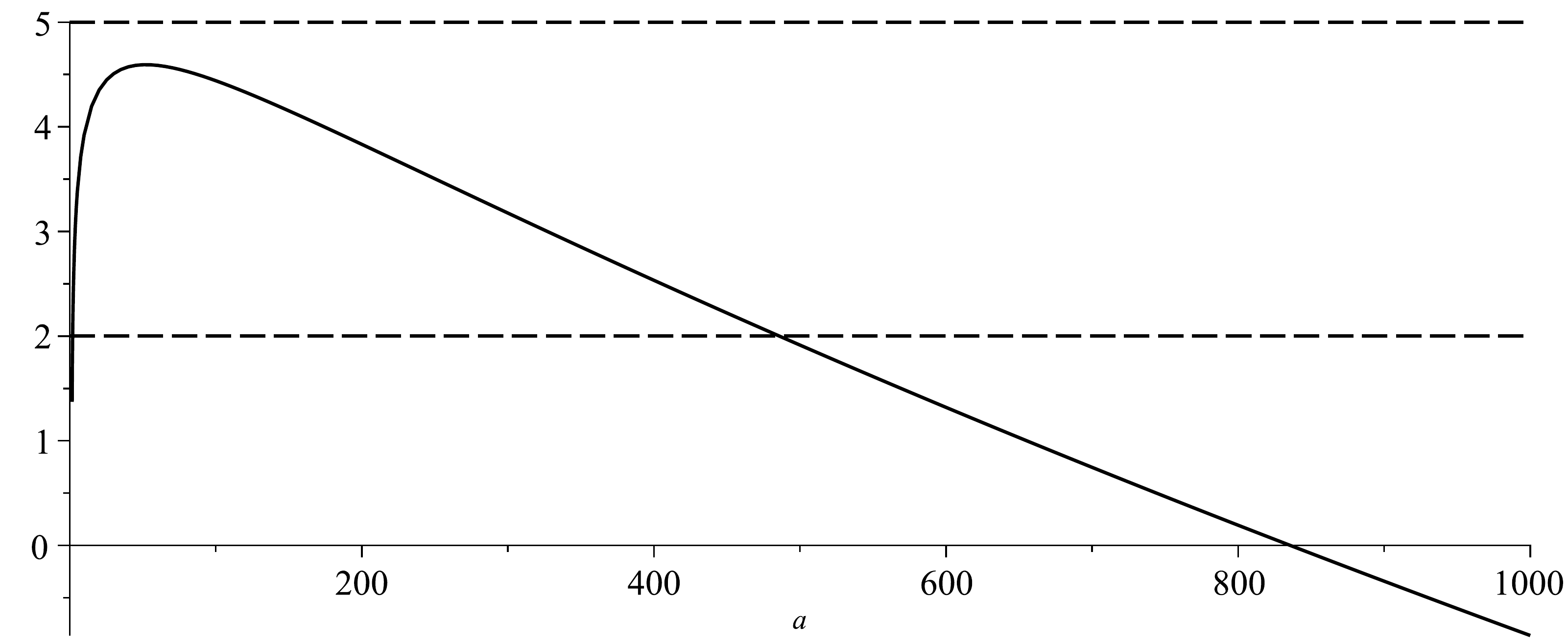}
  \caption{$g(a)$ for $1\leq a\leq 1000$.}
   \label{gPic}
\end{figure}

Since we know $2\leq b+\ell\leq g(a)$, we would like to know what values of $a$ satisfy $2\leq g(a)$.  To find this out, we need to solve $g(a)=2$.  We have decided to omit the solution process here, since it would fill several pages, but, using a computer algebra system, the equation $g(a)=2$ can be converted into a fourth-degree polynomial in $a$, which has four solutions, two real and two complex.  The real solutions are $a=2$ and $a\approx 485.92$.  Since we can see from the graph that $g(a)<2$ for at least one value of $a\geq 486$, we conclude that when $a\geq 486$, $g(a)<2$.  Hence, we must have $a\leq 485$ in order to avoid a contradiction.  This proves the first part of the lemma statement.

In order to prove the second part of the lemma statement, we can simply observe from the graph that for $a\leq 485$, we have $g(a) < 5$.  Since $b+\ell\leq g(a)$, this implies $b+\ell\leq 4$.
\end{proof}

We now present the main result of this section.

\begin{theorem}\label{mu=mu_a,b}
For all $S(a,b)\in\mathcal{Q}_0^\infty$, if $(a,b,n)=(29,1,26)$, $(45,1,33)$, $(47,1,44)$, $(50,1,41)$, $(55,1,50)$, $(67,1,53)$, $(73,1,63)$, or $(79,1,74)$, then
\[\mu_{a,b}(n)=\mu(n)-1.\]
For all other values of $a,b\geq 1$ and $n\in\mathbb{N}_0$ where $n<a$,
\[\mu_{a,b}(n)=\mu(n).\]
\end{theorem}
\begin{proof}
Let $S(a,b)\in\mathcal{Q}_0^\infty$ with $a,b\geq 1$, and suppose $n\in\mathbb{N}_0$ such that $n<a$ and $\mu_{a,b}(n)<\mu(n)$.  Due to Lemma \ref{mu=mu(a,b) small cases}, we know $n\geq 3$, hence we have $a\geq 4$ as well.

Let $\ell$ stand for the amount by which $\mu(n)$ exceeds $\mu_{a,b}(n)$, in other words, $\ell:=\mu(n)-\mu_{a,b}(n)$.  We showed in the last proof that $\mu(n+a)+b+\ell=\mu(n)$, and that $b+\ell\leq 4$.  Since $b,\ell\geq 1$, we have $2\leq \mu(n)-\mu(n+a)\leq 4$.

Since we have an efficient algorithm for calculating values of $\mu$ (Theorem \ref{recursiveExact}), and we know that $a\leq 485$, we can simply calculate all values of $\mu(n)$ and $\mu(n+a)$ for $0\leq n<a\leq 485$, and check which pairs of $(a,n)$ satisfy $2\leq \mu(n)-\mu(n+a)\leq 4$ and which do not.

This search was performed and found that the only pairs satisfying $2\leq \mu(n)-\mu(n+a)\leq 4$ were $(a,n)=(29,26)$, $(45,33)$, $(47,44)$, $(50,41)$, $(55,50)$, $(67,53)$, $(73,63)$, and $(79,74)$.  For each of these pairs, $\mu(n)-\mu(n+a)=2$, which, since $b,\ell\geq 1$, implies that $b=1$ and $\ell=1$.  Hence $\mu_{a,b}(n)=\mu(n)-\ell=\mu(n)-1$.

So far, we have shown that if $\mu_{a,b}(n)<\mu(n)$, then $b=1$ and $(a,n)=(29,26)$, $(45,33)$, $(47,44)$, $(50,41)$, $(55,50)$, $(67,53)$, $(73,63)$, or $(79,74)$.  We will now show the converse, that when $b=1$ and $(a,n)=(29,26)$, $(45,33)$, $(47,44)$, $(50,41)$, $(55,50)$, $(67,53)$, $(73,63)$, or $(79,74)$, we do in fact have $\mu_{a,b}(n)<\mu(n)$.

\begin{center}
\begin{tabular}{ |c|c|c|c| } 
 \hline
 $a$ & $n$ & $\mu(n)$ & $ma+nb$\\ 
 \hline\hline
29 & 26 & 13 & $12(29)+26(1)=374=y_{11}$ \\ \hline
45 & 33 & 15 & $14(45)+33(1)=663=y_{13}$ \\ \hline
47 & 44 & 16 & $15(47)+44(1)=749=y_{14}$ \\ \hline
50 & 41 & 16 & $15(50)+41(1)=791=y_{14}$ \\ \hline
55 & 50 & 17 & $16(55)+50(1)=930=y_{15}$ \\ \hline
67 & 53 & 18 & $17(67)+53(1)=1192=y_{16}$ \\ \hline
73 & 63 & 19 & $18(73)+63(1)=1377=y_{17}$ \\ \hline
79 & 74 & 20 & $19(79)+74(1)=1575=y_{18}$ \\ \hline
\end{tabular}
\end{center}

As is shown in the table above, for the eight exceptional cases of $(a,n)$, there exists $m<\mu(n)$ such that $ma+nb\in S(a,b)$.  Hence, in these cases, $\mu_{a,b}(n)<\mu(n)$, and, as we have already shown, this implies that $\mu_{a,b}(n)=\mu(n)-1$.

\end{proof}

As a consequence of this, we can define the Ap\'ery sets $\Ap(S(a,b),a)$ purely in terms of $\mu$, $a$, and $b$, rather than in terms of $\mu_{a,b}$ as we had previously done.

\begin{corollary}\label{AperyCorollary}
For all $S(a,b)\in\mathcal{Q}_0^\infty$ where $a,b\geq 1$, if $(a,b)\neq (29,1)$, $(45,1)$, $(47,1)$, $(50,1)$, $(55,1)$, $(67,1)$, $(73,1)$, or $(79,1)$, then
\[\Ap(S(a,b),a)=\Set{\mu(n)a+nb | n=0,1,2,\dots,a-1}\]
\end{corollary}

\section{Bounds on the Frobenius Number and the Genus}

As a consequence of Corollary \ref{AperyCorollary} and Lemma \ref{FgLemma}, which gives formulas for the Frobenius number and the genus in terms of an Ap\'ery set, we have the following.

\begin{theorem}\label{FgMuFormula}
For all $S(a,b)\in\mathcal{Q}_0^\infty$ where $a,b\geq 1$ and $(a,b)\neq (29,1)$, $(45,1)$, $(47,1)$, $(50,1)$, $(55,1)$, $(67,1)$, $(73,1)$, or $(79,1)$, the Frobenius number of $S(a,b)$ is given by
\[\F(S(a,b))=\max\Set{\mu(n)a+nb | n=0,1,2,\dots,a-1}-a\]
and the genus of $S(a,b)$ is given by
\[\g(S(a,b))=\frac{(a-1)(b-1)}{2}+\sum_{n=0}^{a-1} \mu(n)\]
\end{theorem}
\begin{proof}
According to Lemma \ref{FgLemma}, $\F(S(a,b))=\max{\Ap(S(a,b),a)}-a$. Also, Corollary \ref{AperyCorollary} says $\Ap(S(a,b),a)=\Set{\mu(n)a+nb | n=0,1,2,\dots,a-1}$ when $(a,b)\neq (29,1)$, $(45,1)$, $(47,1)$, $(50,1)$, $(55,1)$, $(67,1)$, $(73,1)$, or $(79,1)$. Therefore, $\F(S(a,b))=(\max\Set{\mu(n)a+nb | n=0,1,2,\dots,a-1})-a$.

Also according to Lemma 2.6,
\begin{align*}
\g(S(a,b)) &= \frac{1}{a} \left(\sum_{w\in \Ap(S(a,b),a)} w \right)- \frac{a-1}{2} \\
&= \frac{1}{a} \left(\sum_{n=0}^{a-1} \mu(n)a+nb \right)- \frac{a-1}{2} \\
&= \frac{1}{a} \left(\sum_{n=0}^{a-1} \mu(n)a \right) + \frac{1}{a} \left(\sum_{n=0}^{a-1} nb \right)- \frac{a-1}{2} \\
&= \left(\sum_{n=0}^{a-1} \mu(n) \right) + \frac{b}{a} \left(\sum_{n=0}^{a-1} n \right)- \frac{a-1}{2} \\
&= \left(\sum_{n=0}^{a-1} \mu(n) \right) + \frac{b}{a} \left(\frac{a(a-1)}{2} \right)- \frac{a-1}{2} \\
&= \left(\sum_{n=0}^{a-1} \mu(n) \right) + b \left(\frac{a-1}{2} \right)- \frac{a-1}{2} \\
&= \left(\sum_{n=0}^{a-1} \mu(n) \right) + \frac{(a-1)(b-1)}{2}
\end{align*}
\end{proof}

Using the bounds on $\mu$ developed in the last section, together with the formulas for the Frobenius number $\F(S(a,b))$ and the genus $\g(S(a,b))$ which were given in Theorem \ref{FgMuFormula}, we can now give bounds and asymptotic behavior for $\F(S(a,b))$ and $\g(S(a,b))$ in terms of $a$ and $b$.

\begin{theorem}
For all $S(a,b)\in\mathcal{Q}_0^\infty$ where $a,b\geq 1$ and $(a,b)\neq (29,1)$, $(45,1)$, $(47,1)$, $(50,1)$, $(55,1)$, $(67,1)$, $(73,1)$, or $(79,1)$,
\[\F(S(a,b))\geq\frac{a}{2}\left(1+\sqrt{8a-7}\right)+ab-a-b\]
and
\[\F(S(a,b))\leq\frac{a}{2}\left(3+\sqrt{24a-15}\right)+ab-a-b\]
\end{theorem}
\begin{proof}
Recall from Theorem \ref{FgMuFormula} that
\begin{align*}
\F(S(a,b))&=\max\Set{a\mu(n)+nb | n=0,1,2,\dots,a-1}-a \\
&\geq a \mu(a-1)+(a-1)b-a \\
&=a \mu(a-1)+ab-a-b
\end{align*}
Applying the lower bound on $\mu$ from Theorem \ref{fBound}, we get
\begin{align*}
\F(S(a,b))&\geq a f(a-1)+ab-a-b \\
&= \frac{a}{2}\left(1+\sqrt{8a-7}\right)+ab-a-b
\end{align*}

As for the upper bound, we can first break up the maximum as follows.
\begin{align*}
\F(S(a,b))&=\max\Set{a\mu(n)+nb | n=0,1,2,\dots,a-1}-a \\
&\leq \max\Set{a\mu(n) | n=0,\dots,a-1}+\max\Set{nb | n=0,\dots,a-1}-a \\
&=\max\Set{a \mu(n) | n=0,1,2,\dots,a-1}+(a-1)b-a \\
&=\max\Set{a \mu(n) | n=0,1,2,\dots,a-1}+ab-a-b
\end{align*}
We can then use the Gauss upper bound from Theorem \ref{GaussBound} to get
\[\F(S(a,b))\leq \max\Set{a\cdot 3f\left(\frac{n}{3}\right) | n=0,1,2,\dots,a-1}+ab-a-b\]
Since $f$ is an increasing function, this yields
\begin{align*}
\F(S(a,b))&\leq 3a f\left(\frac{a-1}{3}\right)+ab-a-b \\
&=\frac{a}{2}\left(3+\sqrt{24a-15}\right)+ab-a-b
\end{align*}
\end{proof}

As a consequence of these bounds, we can use asymptotic Bachmann-Landau notation to say the following.  Readers unfamiliar with this notation should refer to \cite{knuth76}.

\begin{corollary}
For all $S(a,b)\in\mathcal{Q}_0^\infty$, $\F(S(a,b))=\Theta(a^{3/2}+ab)$.
\end{corollary}

Turning our attention now to the genus, we find the following.

\begin{theorem}
For all $S(a,b)\in\mathcal{Q}_0^\infty$ where $a,b\geq 1$ and $(a,b)\neq (29,1)$, $(45,1)$, $(47,1)$, $(50,1)$, $(55,1)$, $(67,1)$, $(73,1)$, or $(79,1)$,
\[\g(S(a,b))\geq\frac{1}{24}\left((8a-7)^{3/2}+12a-13\right)+\frac{(a-1)(b-1)}{2}\]
and
\[\g(S(a,b))\leq\frac{1}{24}\left(\sqrt{3}(8a+3)^{3/2}+36a-36-11\sqrt{33}\right)+\frac{(a-1)(b-1)}{2}\]
\end{theorem}
\begin{proof}
Recall from Theorem \ref{FgMuFormula} that
\[\g(S(a,b))=\frac{(a-1)(b-1)}{2}+\sum_{n=0}^{a-1} \mu(n)\]
Furthermore, since $\mu(0)=0$, we can drop the $\mu(0)$ term to get
\[\g(S(a,b))=\frac{(a-1)(b-1)}{2}+\sum_{n=1}^{a-1} \mu(n)\]
Hence, applying the lower bound on $\mu(n)$ from Theorem \ref{fBound}, we get
\[\g(S(a,b))\geq\frac{(a-1)(b-1)}{2}+\sum_{n=1}^{a-1} f(n)\]

As $f$ is an increasing function, we can bound this sum by the following definite integral, which yields the formula in the theorem statement.
\[\g(S(a,b))\geq\frac{(a-1)(b-1)}{2}+\int_{n=0}^{a-1} f(n) dn\]

As for the upper bound on $\g(S(a,b))$, we use the same process, but applying the upper Gauss bound on $\mu(n)$ from Theorem \ref{GaussBound}.  Hence
\begin{align*}
\g(S(a,b))&=\frac{(a-1)(b-1)}{2}+\sum_{n=1}^{a-1} \mu(n) \\
&\leq \frac{(a-1)(b-1)}{2}+\sum_{n=1}^{a-1} 3f\left(\frac{n}{3}\right) \\
&\leq \frac{(a-1)(b-1)}{2}+\int_{n=1}^{a} 3f\left(\frac{n}{3}\right) dn
\end{align*}
which yields the upper bound in the theorem statement.
\end{proof}

As a consequence of these bounds, we can once again use asymptotic \linebreak Bachmann-Landau notation to say the following.

\begin{corollary}
For all $S(a,b)\in\mathcal{Q}_0^\infty$, $\g(S(a,b))=\Theta(a^{3/2}+ab)$.
\end{corollary}

\section{The Embedding Dimension}

Since our numerical semigroups are generated by infinite quadratic sequences, the value of the embedding dimension is non-trivial to establish.  In this section, we will build up a sequence of theorems, culminating in a proof of the exact value of the embedding dimension for all $S(a,b)\in\mathcal{Q}_0^\infty$.

We begin by observing that from Lemma \ref{multembedlemma}, the multiplicity of $S(a,b)$ is $\m(S(a,b))=a$, and so the embedding dimension of $S(a,b)$ is bounded above by $\e(S(a,b))\leq a$.  This fact is also a consequence of the following theorem, which gives a more detailed picture than simply that $\e(S(a,b))\leq a$.

\begin{theorem}\label{notMinGenSetThm}
Let $S(a,b)\in\mathcal{Q}_0^\infty$ and let $y_n=na+{n\choose 2}b$ for all $n\in\mathbb{N}_0$.  If $n > a$, then $y_n$ is not in the minimal set of generators of $S(a,b)$.
\end{theorem}
\begin{proof}
Assume $n>a$. Then, $n=a+k$ for some positive integer $k$. So,
\begin{align*}
y_n&=na+\frac{n(n-1)}{2}b \\
&=(a+k)a+\frac{(a+k)(a+k-1)}{2}b \\
&=a^2+ak+\frac{a^2+k^2+2ak-a-k}{2}b.
\end{align*}
Also, we know that $y_a=a^2+\frac{a^2-a}{2}b$ and $y_k=ak+\frac{k^2-k}{2}b$. Therefore, $y_n=y_a+y_k+abk$. Since $y_1=a$, $y_n=y_a+y_k+y_1(bk)$. Also, $bk \in \mathbb{N}_0$ and $1,a,k<n$, so $y_1,y_a,y_k<y_n$. Therefore, $y_n$ is generated by smaller elements of $S$, so $y_n$ is not in the minimal set of generators of $S$.
\end{proof}

On the other hand, we get a lower bound on the embedding dimension from the following.

\begin{theorem}\label{MinSetGenThm}
Let $S(a,b)\in\mathcal{Q}_0^\infty$ and let $y_n=na+{n\choose 2}b$ for all $n\in\mathbb{N}_0$.  If ${n \choose 2}<a$, then $y_n$ is in the minimal set of generators of $S(a,b)$.
\end{theorem}
\begin{proof}
Suppose that ${n \choose 2}<a$ and that $y_n$ is not in the minimal set of generators of $S$. Then there exist some $c_1,c_2,\ldots ,c_{n-1} \in \mathbb{N}_0$ such that
\[ y_n=c_1y_1 + c_2y_2 + \cdots + c_{n-1}y_{n-1}=\sum_{i=1}^{n-1} c_iy_i\]
Plugging in the formula for each $y_i$, we obtain
\[ na+{n \choose 2}b = \sum_{i=1}^{n-1} c_i\left(ia+{i \choose 2}b\right)\]
Rearranging to get all copies of $a$ on the left and all copies of $b$ on the right, we get
\[ \left(n-\sum_{i=1}^{n-1} c_ii\right)a=\left(\left(\sum_{i=1}^{n-1} c_i{i \choose 2}\right)-{n \choose 2}\right)b\]
Since the left hand side is a multiple of $a$, the right hand side must also be a multiple of $a$. Since $a>1$ and $\gcd(a,b)=1$, it must be that $\left(\sum_{i=1}^{n-1} c_i{i \choose 2}\right)-{n \choose 2}$ is a multiple of $a$. Likewise, since the right hand side is a multiple of $b$, the left hand side must also be a multiple of $b$. If $b>1$, then since $\gcd(a,b)=1$, it must be that $n-\sum_{i=1}^{n-1} c_ii$ is a multiple of $b$. On the other hand, if $b=1$, then $n-\sum_{i=1}^{n-1} c_ii$ is a multiple of $b$ by default. Furthermore, in order for the equation to be balanced, these two expressions must be the same multiple of $a$ and $b$ respectively. In other words, there exists some $k \in \mathbb{Z}$ such that
\[ \left(\sum_{i=1}^{n-1} c_i{i \choose 2}\right)-{n \choose 2}=ka\]
\[ n-\sum_{i=1}^{n-1} c_ii=kb\]
If we multiply both sides of the second equation by $\frac{n-1}{2}$, we obtain
\[ {n \choose 2}-\sum_{i=1}^{n-1} c_i\frac{i(n-1)}{2}=\frac{n-1}{2}kb\]
Adding this to the first equation, we get
\[ \left(\sum_{i=1}^{n-1} c_i{i \choose 2}\right)-{n \choose 2}+{n \choose 2}-\sum_{i=1}^{n-1} c_i\frac{i(n-1)}{2}=ka+\frac{n-1}{2}kb\]
Collecting each copy of $c_i$ together and cancelling the ${n \choose 2}$ terms, we get
\[ \sum_{i=1}^{n-1} c_i\left({i \choose 2}-\frac{i(n-1)}{2}\right)=ka+\frac{n-1}{2}kb\]
Simplifying, we obtain
\[ \sum_{i=1}^{n-1} c_i\left(\frac{i(i-1)}{2}-\frac{i(n-1)}{2}\right)=ka+\frac{n-1}{2}kb\]
\[ \sum_{i=1}^{n-1} c_i\left(\frac{i(i-n)}{2}\right)=ka+\frac{n-1}{2}kb\]
Each $i-n$ must be negative, hence the entire left hand side must be negative, since each $c_i\geq 0$ and not all of them can be 0. Therefore,
\[ ka+\frac{n-1}{2}kb<0\]
\[ k\left(a+\frac{n-1}{2}b\right)<0\]
Since $a+\frac{n-1}{2}b$ is clearly positive, we conclude that $k<0$. Recall that
\[ \left(\sum_{i=1}^{n-1} c_i{i \choose 2}\right)-{n \choose 2}=ka\]
Since $k<0$, this implies $k\leq -1$, so
\[ \left(\sum_{i=1}^{n-1} c_i{i \choose 2}\right)-{n \choose 2}\leq -a\]
Rearranging, we obtain
\[ {n \choose 2}\geq a+\sum_{i=1}^{n-1} c_i{i \choose 2}\]
Since every $c_i\geq 0$, this implies ${n \choose 2}\geq a$. However, we began by assuming that ${n \choose 2}<a$. This is a contradiction, so $y_n$ is in the minimal set of generators of $S$ when ${n \choose 2}<a$.
\end{proof}

When the value of $a$ is fixed, we can say a bit more about the set of minimal generators for different values of $b$.

\begin{theorem}\label{bPlusIThm}
Let $S(a,b)\in\mathcal{Q}_0^\infty$ and $S(a,b+i)\in\mathcal{Q}_0^\infty$ for some $i\in\mathbb{N}_0$.  For all $n\in\mathbb{N}_0$, if $na+{n \choose 2}b$ is not in the minimal set of generators of $S(a,b)$, then $na+{n \choose 2}(b+i)$ is not in the minimal set of generators of $S(a,b+i)$.
\end{theorem}
\begin{proof}
Let us call the generators of $S(a,b)$ by the names $y_\ell $ and the generators of $S(a,b+i)$ by the names $y^\prime_\ell $, so that $y_\ell =\ell a+{\ell  \choose 2}b$, and $y^\prime_\ell =\ell a+{\ell  \choose 2}(b+i)$ for all $\ell \in\mathbb{N}_0$. Suppose that $y_n=na+{n \choose 2}b$ is not in the minimal set of generators of $S(a,b)$. Then there exist some $c_1,c_2,\ldots ,c_{n-1}\in\mathbb{N}_0$ such that
\[ y_n=c_1y_1+c_2y_2+\cdots +c_{n-1}y_{n-1}=\sum_{j=1}^{n-1} c_jy_j\]
Plugging in the formula for each $y_\ell $, we obtain
\[ na+{n \choose 2}b=\sum_{j=1}^{n-1} c_j\left(ja+{j \choose 2}b\right)\]
By following the same technique as in the proof of Theorem \ref{MinSetGenThm}, we can move all copies of $a$ to one side and all copies of $b$ to the other. Then both sides with be multiples of both $a$ and $b$, and since $\gcd(a,b)=1$, $\left(\sum_{j=1}^{n-1} c_j{j \choose 2}\right)-{n \choose 2}$ is a multiple of $a$. So,
\[ \left(\sum_{j=1}^{n-1} c_j{j \choose 2}\right)-{n \choose 2}=ka\]
where $k\in\mathbb{Z}$ and $k<0$. Hence
\[ {n \choose 2}=\left(\sum_{j=1}^{n-1} c_j{j \choose 2}\right)-ka\]
Multiply this by $i$ to get
\[ {n \choose 2}i=\left(\sum_{j=1}^{n-1} c_j{j \choose 2}i\right)-kai\]
Now, consider these two equations:
\[ na+{n \choose 2}b=\sum_{j=1}^{n-1} c_j\left(ja+{j \choose 2}b\right)\]
\[ {n \choose 2}i=\left(\sum_{j=1}^{n-1} c_j{j \choose 2}i\right)-kai\]
If we add these two equations together, we get
\[ na+{n \choose 2}b+{n \choose 2}i=\left(\sum_{j=1}^{n-1} c_j\left(ja+{j \choose 2}b\right)\right)+\left(\sum_{j=1}^{n-1} c_j{j \choose 2}i\right)-kai\]
Collecting terms, we get
\[ na+{n \choose 2}(b+i)=\left(\sum_{j=1}^{n-1} c_j\left(ja+{j \choose 2}(b+i)\right)\right)-kai\]
Replacing with $y^\prime_\ell $ wherever possible, this becomes
\[ y^\prime_n=\left(\sum_{j=1}^{n-1} c_jy^\prime_j\right)-ky^\prime_1i\]
Since $k<0$, this shows that $y^\prime_n$ is a linear combination over $\mathbb{N}_0$ of smaller generators. Hence $y^\prime_n=na+{n \choose 2}(b+i)$ is not in the minimal set of generators for $S(a,b+i)$.
\end{proof}

As a consequence of the previous theorem, the embedding dimension is decreasing on $b$ when $a$ is fixed.

\begin{corollary}
For all $S(a,b_1),S(a,b_2)\in\mathcal{Q}_0^\infty$, if $b_1\leq b_2$, then $\e(S(a,b_1))\geq\e(S(a,b_2))$.
\end{corollary}

So far, we have established the following about the generators $y_n$ of $S(a,b)\in\mathcal{Q}_0^\infty$.  If $n>a$, then $y_n$ is not in the minimal set of generators of $S(a,b)$, from Theorem \ref{notMinGenSetThm}.  If ${n \choose 2}<a$, then $y_n$ is in the minimal set of generators of $S(a,b)$, from Theorem \ref{MinSetGenThm}.  Also, it is obvious that if $n\choose 2$ is a multiple of $a$, then $y_n$ is a multiple of $y_1=a$, so $y_n$ is not in the minimal set of generators of $S(a,b)$.

Thus, it remains only to investigate the generators $y_n$ of $S(a,b)$ where $a<{n\choose 2}\leq {a\choose 2}$ and $n\choose 2$ is not a multiple of $a$.  We will do this in the proof of the next theorem.

\begin{theorem}\label{n choose 2 > a}
For all $S(a,b)\in\mathcal{Q}_0^\infty$ with $a,b\geq 1$, if ${n\choose 2}>a$ and $(a,b,{n\choose 2}\mod a)\neq (29,1,26)$, $(45,1,33)$, $(47,1,44)$, $(50,1,41)$, $(55,1,50)$, $(67,1,53)$, $(73,1,63)$, $(79,1,74)$, then $na+{n\choose 2}b$ is not in the set of minimal generators of $S(a,b)$.
\end{theorem}
\begin{proof}
Let $S(a,b)\in\mathcal{Q}_0^\infty$ with $a,b\geq 1$.  Let $y_n$ denote the generators $y_n=na+{n \choose 2}b$ for all $n\in\mathbb{N}_0$.  Suppose that ${n\choose 2}>a$ and $(a,b,{n\choose 2}\mod a)\neq (29,1,26)$, $(45,1,33)$, $(47,1,44)$, $(50,1,41)$, $(55,1,50)$, $(67,1,53)$, $(73,1,63)$, $(79,1,74)$ and that $y_n$ is in the minimal set of generators of $S(a,b)$.

We already know that if $n>a$ or $n\choose 2$ is a multiple of $a$, then $y_n$ is not in the minimal set of generators of $S(a,b)$.  Hence $a<{n\choose 2}\leq {a\choose 2}$ and $n\choose 2$ is not a multiple of $a$.  Then there exist $j,k\in\mathbb{N}_0$ such that ${n\choose 2}=ka+j$ and $k\geq 1$ and $1\leq j\leq a-1$.  Note that $j={n\choose 2}\mod a$.  Then
\[y_n=na+{n\choose 2}b=na+(ka+j)b=(n+kb)a+jb\]

The Ap\'ery set $\Ap(S(a,b),a)$, together with $a$, forms a generating set for $S(a,b)$, so it must contain the minimal generating set of $S(a,b)$.  Hence $y_n\in\Ap(S(a,b),a)$.  We established in Theorem \ref{AperyS(a,b)} that 
\[\Ap(S(a,b),a)=\Set{\mu_{a,b}(j)a+jb | j=0,1,2,\dots,a-1}\]
Hence, $y_n=\mu_{a,b}(j)a+jb$, and so $\mu_{a,b}(j)=n+kb$.

Furthermore, since we have assumed that $(a,b,{n\choose 2}\mod a)=(a,b,j)\neq (29,1,26)$, $(45,1,33)$, $(47,1,44)$, $(50,1,41)$, $(55,1,50)$, $(67,1,53)$, $(73,1,63)$, $(79,1,74)$, Theorem \ref{mu=mu_a,b} tells us that $\mu_{a,b}(j)=\mu(j)$, hence $\mu(j)=n+kb$.

We know from the Gauss bound (Theorem \ref{GaussBound}) that $\mu(j)\leq 3f\left(\frac{j}{3}\right)$.  Hence $n+kb\leq 3f\left(\frac{j}{3}\right)$.  We can also deduce from ${n\choose 2}=ka+j$ that $n=f(ka+j)$ and plug this into the inequality to get $f(ka+j)+kb\leq 3f\left(\frac{j}{3}\right)$.  Since $k,b\geq 1$, we can loosen this inequality to say that $f(ka+j)+1\leq 3f\left(\frac{j}{3}\right)$.  Now we will isolate $k$ from this inequality.
\begin{align*}
    f(ka+j)+1&\leq 3f\left(\frac{j}{3}\right) \\
    \frac{1+\sqrt{8(ka+j)+1}}{2}+1&\leq \frac{3\left(1+\sqrt{\frac{8j}{3}+1}\right)}{2} \\
    3+\sqrt{8(ka+j)+1}&\leq 3\left(1+\sqrt{\frac{8j}{3}+1}\right) \\
    \sqrt{8(ka+j)+1}&\leq 3\sqrt{\frac{8j}{3}+1} \\
    8(ka+j)+1 &\leq 9\left(\frac{8j}{3}+1\right) \\
    8ka+8j+1 &\leq 24j+9 \\
    8ka &\leq 16j+8 \\
    ka &\leq 2j+1
\end{align*}

Recall that $j\leq a-1$, hence $ka\leq 2(a-1)+1=2a-1$, which implies $k<2$.  Since we already know $k\geq 1$, we must have $k=1$.

With the value of $k$ now known, the earlier equation $\mu(j)=n+kb$ becomes $\mu(j)=n+b$, or, when rearranged, $b=\mu(j)-n$.  This seemingly innocuous statement is actually quite powerful, because the right hand side, $\mu(j)-n$, depends only on the values of $n$ and $a$ (recall that $j={n \choose 2}\mod a$), not on the value of $b$.

Thus, we conclude that, for a given, fixed value of $a$, if $a<{n\choose 2}\leq {a\choose 2}$ and $n\choose 2$ is not a multiple of $a$, then there is at most one value of $b$ for which $y_n$ is in the minimal generating set of $S(a,b)$.

There are two cases.  If $y_n$ is in the minimal set of generators of $S(a,1)$, then $y_n$ is not in the minimal set of generators for any other $b\geq 2$, due to the considerations in the last paragraph.  On the other hand, if $y_n$ is not in the minimal set of generators of $S(a,1)$, then, due to Theorem \ref{bPlusIThm}, we still have that $y_n$ is not in the minimal set of generators of $S(a,b)$ where $b\geq 2$.

Now let us consider what happens when $b=1$.  Recall from earlier that we know $k=1$, so the earlier equation $\mu(j)=n+kb$ becomes $\mu(j)=n+1$.  Also, $n=f(ka+j)$ becomes $n=f(a+j)$.  Hence $f(a+j)+1=\mu(j)$.  Since $j\geq 1$, we can apply the combined bound (Theorem \ref{combined bound}) to $\mu(j)$ to get
\[f(a+j)+1\leq f(j)+3f\left(\frac{f(j)-2}{3}\right)\]
Rearranged, this is
\[1\leq f(j)-f(a+j)+3f\left(\frac{f(j)-2}{3}\right)\]

Recall that we saw the quantity on the right-hand side of this inequality in the proof of Lemma \ref{a<=485 lemma}, although we had $n$ in place of $j$ in that proof.  At that time, we showed that it is increasing on $j$.  Hence, since we know $j\leq a-1$,
\[1\leq f(a-1)-f(2a-1)+3f\left(\frac{f(a-1)-2}{3}\right)\]

Recall that we also saw the quantity on the right-hand side of this inequality in the proof of Lemma \ref{a<=485 lemma}, and we defined it as the function $g(a)$, which is shown in Figure \ref{gPic} (located in the proof of Lemma \ref{a<=485 lemma}).

In order to determine what values of $a$ will satisfy $g(a)\geq 1$, we need to solve $g(a)=1$.  This equation can be converted to a fourth-degree polynomial, and solved exactly.  However, the solution would fill several pages, as would the solutions themselves, so we have decided to omit that work.  The solutions can be checked with any standard computer algebra system.  There are two real solutions, which are $a\approx 1.55$ and $a\approx 655.24$.  We can see from the graph of $g(a)$ that there is at least one value of $a>655$ for which $g(a)<1$, hence, we must have $a\leq 655$ in order to have $g(a)\geq 1$.

Since we have an efficient algorithm for calculating the values of $\mu$ (Theorem \ref{recursiveExact}), we can simply check all values of $1\leq n\leq a\leq 655$ to see which pairs $(a,n)$ satisfy the equation
\[n+1=\mu\left({n \choose 2}\mod a\right).\]

After performing this check, we found that the only pairs $(a,n)$ with $1\leq n\leq a\leq 655$ and satisfying the above equation are those shown in the table below.  However, as shown in the third column, $y_n$ is not in the minimal set of generators of $S(a,1)$ in any of these cases, because it can be written as a linear combination over $\mathbb{N}_0$ of smaller generators.

\begin{center}
\begin{tabular}{ |l|l|l| } 
 \hline
 $a$ & $n$ & $y_n=c_1 y_1+c_2 y_2+\cdots$ \\ 
 \hline\hline
10 & 6 & $y_6=2y_2+y_3$ \\ \hline
13 & 7 & $y_7=2y_2+y_4$ \\ \hline
19 & 9 & $y_9=2y_2+y_6$ \\ \hline
22 & 9 & $y_9=y_2+y_3+y_5$ \\ \hline
26 & 10 & $y_{10}=y_2+y_3+y_6$ \\ \hline
34 & 12 & $y_{12}=y_2+y_3+y_8$ \\ \hline
40 & 12 & $y_{12}=y_2+y_5+y_6$ \\ \hline
43 & 13 & $y_{13}=y_2+y_4+y_8$ \\ \hline
53 & 15 & $y_{15}=y_2+y_4+y_{10}$ \\ \hline
58 & 14 & $y_{14}=2y_2+y_3+y_8$ \\ \hline
61 & 15 & $y_{15}=y_2+y_6+y_8$ \\ \hline
64 & 15 & $y_{15}=2y_2+y_3+y_9$ \\ \hline
66 & 16 & $y_{16}=y_3+y_4+y_{10}$ \\ \hline
70 & 16 & $y_{16}=2y_2+y_3+y_{10}$ \\ \hline
78 & 18 & $y_{18}=y_3+y_4+y_{12}$ \\ \hline
82 & 18 & $y_{18}=2y_2+y_3+y_{12}$ \\ \hline
83 & 17 & $y_{17}=2y_2+y_4+y_{10}$ \\ \hline
90 & 18 & $y_{18}=2y_2+y_4+y_{11}$ \\ \hline
97 & 19 & $y_{19}=2y_2+y_4+y_{12}$ \\ \hline
104 & 20 & $y_{20}=2y_2+y_4+y_{13}$ \\ \hline
106 & 21 & $y_{21}=y_3+y_5+y_{14}$ \\ \hline
107 & 21 & $y_{21}=2y_4+y_{14}$ \\ \hline
118 & 22 & $y_{22}=2y_2+y_4+y_{15}$ \\ \hline
142 & 24 & $y_{24}=y_2+y_8+y_{15}$ \\ \hline
181 & 27 & $y_{27}=2y_2+y_6+y_{18}$ \\ \hline
184 & 27 & $y_{27}=y_2+y_3+y_5+y_{18}$ \\ \hline
190 & 28 & $y_{28}=2y_2+y_6+y_{19}$ \\ \hline
193 & 28 & $y_{28}=y_2+y_3+y_5+y_{19}$ \\ \hline
226 & 30 & $y_{30}=y_2+y_3+y_6+y_{20}$ \\ \hline
236 & 31 & $y_{31}=y_2+y_3+y_6+y_{21}$ \\ \hline
\end{tabular}
\end{center}

The theorem is thus proved by contradiction.
\end{proof}

In the next theorem, we will cover the eight exceptional pairs $(a,b)$ that were excluded in the hypothesis of the previous theorem.

\begin{theorem}\label{n choose 2 > a exceptions}
For all $S(a,b)\in\mathcal{Q}_0^\infty$ with $a,b\geq 1$, if $(a,b)= (29,1)$, $(45,1)$, $(47,1)$, $(50,1)$, $(55,1)$, $(67,1)$, $(73,1)$, or $(79,1)$, then there is exactly one $n\in\mathbb{N}$ such that ${n\choose 2}>a$ and $na+{n\choose 2}b$ is in the set of minimal generators of $S(a,b)$.
\end{theorem}
\begin{proof}
Let $S(a,b)\in\mathcal{Q}_0^\infty$ with $a,b\geq 1$.  Let $y_n$ denote the generators $y_n=na+{n \choose 2}b$ for all $n\in\mathbb{N}_0$.  Suppose that $(a,b)= (29,1)$, $(45,1)$, $(47,1)$, $(50,1)$, $(55,1)$, $(67,1)$, $(73,1)$, or $(79,1)$.

We know from the previous theorem that if ${n\choose 2}>a$ and $(a,b,{n\choose 2}\mod a)\neq (29,1,26)$, $(45,1,33)$, $(47,1,44)$, $(50,1,41)$, $(55,1,50)$, $(67,1,53)$, $(73,1,63)$, $(79,1,74)$, then $y_n$ is not in the minimal set of generators of $S(a,b)$.  We also know from Theorem \ref{notMinGenSetThm} that if $n>a$, then $y_n$ is not in the minimal set of generators.

Thus, suppose that $a<{n\choose 2}\leq {a\choose 2}$ and $(a,b,{n\choose 2}\mod a)= (29,1,26)$, $(45,1,33)$, $(47,1,44)$, $(50,1,41)$, $(55,1,50)$, $(67,1,53)$, $(73,1,63)$, or $(79,1,74)$.  We can find all pairs $(a,n)$ that meet this description with a straightforward computation.  They are shown in the table below.  For those values of $(a,n)$ for which we found a way to write $y_n$ as a linear combination over $\mathbb{N}_0$ of smaller generators, we have also listed that.

\begin{center}
\begin{tabular}{ |l|l|l| } 
 \hline
 $a$ & $n$ & $y_n=c_1 y_1+c_2 y_2+\cdots$ \\ 
 \hline\hline
29 & 11 &  \\ \hline
29 & 19 & $y_{19}=y_1+y_2+y_8+y_{11}$ \\ \hline
45 & 13 &  \\ \hline
45 & 33 & $y_{33}=6y_1+2y_2+y_5+2y_{13}$ \\ \hline
47 & 14 &  \\ \hline
47 & 34 & $y_{34}=11y_1+y_3+2y_{14}$ \\ \hline
50 & 14 &  \\ \hline
50 & 39 & $y_{39}=y_1+y_2+y_3+y_5+y_{10}+2y_{14}$ \\ \hline
55 & 15 &  \\ \hline
55 & 26 & $y_{26}=15y_1+y_{15}$ \\ \hline
55 & 30 & $y_{30}=5y_1+y_5+y_{10}+y_{15}$ \\ \hline
55 & 41 & $y_{41}=y_5+3y_{15}$ \\ \hline
67 & 16 &  \\ \hline
67 & 52 & $y_{52}=10y_1+y_8+3y_{16}$ \\ \hline
73 & 17 &  \\ \hline
73 & 57 & $y_{57}=
9y_1+2y_2+y_3+y_6+3y_{17}$ \\ \hline
79 & 18 &  \\ \hline
79 & 62 & $y_{62}=y_6+4y_{18}$ \\ \hline

\end{tabular}
\end{center}

Hence, for all $S(a,b)$ where $(a,b)= (29,1)$, $(45,1)$, $(47,1)$, $(50,1)$, $(55,1)$, $(67,1)$, $(73,1)$, or $(79,1)$, there is at most one $n\in\mathbb{N}$ such that ${n\choose 2}>a$ and $na+{n\choose 2}b$ is in the set of minimal generators of $S(a,b)$.

We will now show that for $(a,n)=(29,11)$, $(45,13)$, $(47,14)$, $(50,14)$, $(55,15)$, $(67,16)$, $(73,17)$, and $(79,18)$, which in the table above have a blank entry in the third column, $y_n$ is in the minimal set of generators of $S(a,1)$.  We will show this by contradiction.

Suppose that $y_n$ is not in the minimal set of generators of $S(a,1)$, where $(a,n)=(29,11)$, $(45,13)$, $(47,14)$, $(50,14)$, $(55,15)$, $(67,16)$, $(73,17)$, or $(79,18)$.  Then there exist some $c_1,c_2,\dots,c_{n-1}\in \mathbb{N}_0$ such that
\[y_n=c_1 y_1+c_2 y_2+\cdots+c_{n-1} y_{n-1}\]

Then for $i=1,2,\dots,n-1$, $y_n\geq c_i y_i$, hence $c_i\leq y_n/y_i$.  Since there are only a finite number of values possible for each $c_i$, and there are only a finite number of terms in the sum, it is straightforward, although tedious, to check that for each pair $(a,n)$, every choice of $c_1,c_2,\dots,c_{n-1}$ such that $c_i\leq y_n/y_i$ for each $i=1,2,\dots,n-1$ has the property that
\[y_n\neq c_1 y_1+c_2 y_2+\cdots+c_{n-1} y_{n-1}\]

We have omitted the process of checking each choice of coefficients, but, having performed the check, we conclude that for $(a,n)=(29,11)$, $(45,13)$, $(47,14)$, $(50,14)$, $(55,15)$, $(67,16)$, $(73,17)$, or $(79,18)$, $y_n$ is in the minimal set of generators of $S(a,1)$.

\end{proof}

We are now prepared to establish the exact embedding dimension of all numerical semigroups generated by an infinite quadratic sequence with initial term zero.

\begin{theorem}
For all $S(a,b)\in\mathcal{Q}_0^\infty$, if $(a,b)=(29,1)$, $(45,1)$, $(47,1)$, $(50,1)$, $(55,1)$, $(67,1)$, $(73,1)$, or $(79,1)$, then
\[\e(S(a,b))=\lceil f(a)\rceil.\]
For all other values of $a,b\geq 1$,
\[\e(S(a,b))=\lceil f(a)\rceil-1.\]
\end{theorem}
\begin{proof}
We will deal with the eight exceptional cases at the end of the proof, so assume that $(a,b)\neq(29,1)$, $(45,1)$, $(47,1)$, $(50,1)$, $(55,1)$, $(67,1)$, $(73,1)$, or $(79,1)$.

We know from Theorems \ref{notMinGenSetThm}, \ref{MinSetGenThm}, and \ref{n choose 2 > a}, that the minimal set of generators of $S(a,b)$ is the set of generators $y_n$ such that ${n\choose 2} < a$.

We now simply need to count how many values of $n$ meet this description in order to get the embedding dimension.  The number of values will be the largest value of $n$ meeting this description.

We know that for $n\geq 1$, ${n\choose 2}<a$ if and only if $n<f(a)$.  Let $N$ stand for the largest value of $n\in\mathbb{N}$ such that ${n\choose 2} < a$.  There are two cases.  If $f(a)$ is an integer, then $N=f(a)-1=\lceil f(a)\rceil -1$.  On the other hand, if $f(a)$ is not an integer, then $N=\lfloor f(a)\rfloor=\lceil f(a)\rceil -1$.  In either case, this tells us that the number of elements $n\in\mathbb{N}$ such that ${n\choose 2}<a$ is $\lceil f(a)\rceil -1$, hence, $\e(S(a,b))=\lceil f(a)\rceil -1$.

Now let us consider those possible eight exceptions, and assume that $(a,b)=(29,1)$, $(45,1)$, $(47,1)$, $(50,1)$, $(55,1)$, $(67,1)$, $(73,1)$, or $(79,1)$.  From Theorems \ref{notMinGenSetThm}, \ref{MinSetGenThm}, and \ref{n choose 2 > a}, the minimal set of generators of $S(a,b)$ contains the generators $y_n$ such that ${n\choose 2} < a$, and Theorem \ref{n choose 2 > a exceptions} tells us that there is exactly one element $y_n$ in the minimal set of generators such that ${n\choose 2} > a$.  Hence, using the same counting arguments as before, $\e(S(a,b))=\lceil f(a)\rceil$.

\end{proof}

\section{Conclusions and Future Work}

In this article, we have investigated all numerical semigroups generated by an infinite quadratic sequence with initial term zero.  We have shown a computationally efficient way to calculate the Ap\'ery set, and given bounds on the elements of the Ap\'ery set, which led to bounds on the genus and the Frobenius number.  With those bounds, we were able to find the asymptotic behavior of the genus and the Frobenius number in terms of the coefficients of the quadratic sequence.  Furthermore, we determined the exact embedding dimension for all such numerical semigroups.

However, there still remain many interesting and unanswered questions.  For numerical semigroups generated by an infinite quadratic sequence with initial term zero, there is still much that could be done.  We could tighten the bounds on $\mu$, or on the Frobenius number, the genus, or the embedding dimension.  We could investigate other associated sets or parameters that we did not mention here, such as the pseudo-Frobenius numbers or the type.

Furthermore, there is much yet to be done in the area of quadratic sequences more generally.  We could investigate infinite quadratic sequences which begin with initial terms other than zero, or we could investigate the effect of using finite numbers of generators.

Lastly, and most importantly, the realm of polynomial sequences of higher degree remains almost completely unexplored.  In this article, we have planted our flag on the edge of that terrain, but an infinite expanse yet remains to be surveyed.

\printbibliography

\end{document}